\documentclass{article}
\usepackage{amsmath,amscd,amstext,amssymb,amsfonts}
\usepackage{amsthm}
\usepackage{graphics}
\usepackage{latexsym,empheq}
\usepackage{color,xcolor}
\usepackage{bookman}
\usepackage{bm}
\usepackage{cite}
\usepackage{enumerate}

\newcommand{\nd}{\ensuremath{{\color{red}d}}} 
\renewcommand{\nd}{\ensuremath{{d}}} 

\newcommand{\beq}{\begin{eqnarray}}
\newcommand{\eeq}{\end{eqnarray}}
\newcommand{\beqq}{\begin{eqnarray*}}
\newcommand{\eeqq}{\end{eqnarray*}}
\newcommand{\bit}{\begin{itemize}}
\newcommand{\eit}{\end{itemize}}
\newcommand{\benu}{\begin{enumerate}}
\newcommand{\eenu}{\end{enumerate}}
\newcommand{\ds}{\displaystyle}
\newcommand{\R}{\ensuremath{\mathbb R}}\newcommand{\real}{\R}
\newcommand{\Rn}{\ensuremath{{\mathbb R}^{\nd}}} \newcommand{\rn}{\Rn}
\newcommand{\N}{{\ensuremath{\mathbb N}}}\newcommand{\nat}{\N}
\newcommand{\No}{\ensuremath{\N_{0}}}\newcommand{\no}{\N_0}
\newcommand{\Z}{\mathbb Z} 
\newcommand{\Zn}{\Z^{\nd}}

\newcommand{\cD}{\mathcal{D}}
\newcommand{\Dd}{\mathrm{D}}
\newcommand{\dint}{\;\mathrm{d}}

\newcommand{\llqe}{\ensuremath{\ell_{q_1}(\beta_j \ell_{p_1}^{M_j})}}
\newcommand{\llqz}{\ensuremath{\ell_{q_2}(\ell_{p_2}^{M_j})}}

\newcommand{\A}{\ensuremath{A^s_{p,q}}}  
\newcommand{\Ae}{\ensuremath{A^{s_1}_{p_1,q_1}}}  
\newcommand{\Az}{\ensuremath{A^{s_2}_{p_2,q_2}}}  
\newcommand{\B}{\ensuremath{B^s_{p,q}}}  
\newcommand{\Bsigmae}{\ensuremath{B^\sigma_{p_1,q_1}}}  
\newcommand{\Bsigmaz}{\ensuremath{B^\tau_{p_2,q_2}}}  
\newcommand{\Fsigmae}{\ensuremath{F^\sigma_{p_1,q_1}}}  
\newcommand{\Fsigmaz}{\ensuremath{F^\tau_{p_2,q_2}}}

\newcommand{\be}{\ensuremath{B^{s_1}_{p_1,q_1}}}  
\newcommand{\bz}{\ensuremath{B^{s_2}_{p_2,q_2}}}  
\newcommand{\F}{\ensuremath{F^s_{p,q}}}

\def\supp{\mathop{\rm supp}\nolimits}
\newcommand{\bpr}{\begin{proof}}
\newcommand{\epr}{\end{proof}}
\newcommand{\bli}{\begin{list}{}{\labelwidth6mm\leftmargin8mm}}
\newcommand{\eli}{\end{list}}

\def\id{\mathop{\rm id}\nolimits}
\def\ext{\mathop{\rm ext}\nolimits}
\def\re{\mathop{\rm re}\nolimits}

\newcommand{\nn}[1]{\ensuremath \nu( #1)}
\newcommand{\tn}{\ensuremath\mathbf{t}}

\allowdisplaybreaks[1]
\numberwithin{equation}{section}

\newtheorem{lemma}{Lemma}[section]
\newtheorem{corollary}[lemma]{Corollary}
\newtheorem{proposition}[lemma]{Proposition}
\newtheorem{theorem}[lemma]{Theorem}

\theoremstyle{definition}
\newtheorem{definition}[lemma]{Definition}
\newtheorem{example}[lemma]{Example}
\newtheorem{Exams}[lemma]{Examples}
\newtheorem{Exam}[lemma]{Example}

\theoremstyle{remark}
\newtheorem{remark}[lemma]{Remark}

\begin{document}

\title{Nuclear and compact embeddings in function spaces of generalised smoothness}

\author{Dorothee D. Haroske\footnotemark[1], \ Hans-Gerd Leopold, \ Susana D. Moura\footnotemark[2], \\ and Leszek Skrzypczak\footnotemark[1]}


\date{\slshape Dedicated to Oleg V. Besov on the occasion of his 90th birthday}

\maketitle

\footnotetext[1]{Both authors were partially supported by the German Research Foundation (DFG), Grant no. Ha 2794/8-1.}
\footnotetext[2]{The author was partially supported by the Center for Mathematics of the University of Coimbra-UIDB/00324/2020, funded by the Portuguese Government through FCT/MCTES.}

{

{  \begin{abstract}
    We study nuclear embeddings for function spaces of generalised smoothness defined on a bounded Lipschitz domain $\Omega\subset\Rn$. This covers, in particular, the well-known situation for spaces of Besov and Triebel-Lizorkin spaces defined on bounded domains as well as some first results for function spaces of logarithmic smoothness. In addition, we provide some new, more general approach to compact embeddings for such function spaces, which also unifies earlier results  in different settings, including also the study of their entropy numbers. Again we rely on suitable wavelet decomposition techniques and the famous Tong result (1969) about nuclear diagonal operators acting in $\ell_r$ spaces, which we could recently extend to the vector-valued setting needed here. 
  \end{abstract}

  

  \section{Introduction} \label{sect-intro}
  Our main intention in this paper is to characterise nuclear embeddings between function spaces of generalised smoothness. We were motivated by the nice paper \cite{CoDoKu}, but now we are able to deal with the topic in a much more general setting sealing also some small gaps left open there. In other words, we 
  combine two well-established concepts in functional analysis, that is, the theory of function spaces of generalised smoothness as well as the theory of nuclear operators, benefit from some rather recent methods (wavelet decompositions, sequence space results) and finally obtain a criterion for the nuclearity of continuous embedding operators of the type
  \begin{equation}\label{i-0}
    \id_\Omega : \Bsigmae(\Omega) \hookrightarrow \Bsigmaz(\Omega).
    \end{equation}
Here $\Omega$ is a bounded Lipschitz domain and $B^\sigma_{p,q}(\Omega)$ are Besov spaces of generalised smoothness defined by restriction to $\Omega$, where $\sigma = (\sigma_j)$ represents the smoothness, and $1\leq p,q\leq \infty$. As far as we know this question has not been studied before apart from some special cases. Our proof turns out to be surprisingly simple and straightforward, once all the necessary preparations are available.

The study of spaces of generalised smoothness has a long history, resulting on
one hand from the interpolation side
(with a function parameter), see  \cite{Me} and \cite{CoFe}, whereas the rather abstract approach
(approximation by series of entire analytic functions and coverings as well as by differences and moduli of continuity) was
independently developed by {Gol'dman} and {Kalyabin} in the late 70's
and early 80's of the last century; we refer to the
survey \cite{KaLi87} and the appendix \cite{lizorkin} which cover the
extensive (Russian) literature at that time. We rely on the Fourier-analytical
approach as presented in \cite{BrMo03,CaFa04,CaLe1,FaLe,Mo1,Mo2,Mo3}. There one can also find further motivation for the continuous interest in the study of such spaces: its connection with applications for pseudo-differential operators (as generators of sub-Markovian semi-groups) or the study of trace
spaces on fractals, that is, so-called $h$-sets $\Gamma$. Plainly such  applications and also the topic in its full generality are out of the scope of the present paper.

It is well known that a generalisation of the Besov spaces $\B(\rn)$ and the Triebel-Lizorkin spaces $\F(\rn)$ can be given by a more general weight sequence $(\sigma_j)_{j\in\No}$ instead of the sequence $(2^{js})_{j\in\No}$, $s\in\real$, representing the smoothness, and an almost strongly increasing sequence $(N_j)_{j\in\No}$ instead of the dyadic sequence $(2^j)_{j\in\No}$ in the definition of the resolution of unity. We refer to \cite{FaLe} for the history of these spaces and further details. 
Thanks to the {\em standardization procedure} used in \cite{CaLe2} one can restrict the more general setting of partitions of $\rn$ to the dyadic one, at the expenses of adapted smoothness sequences.

The purpose of the present paper is to compare and unify different approaches to spaces of generalised smoothness, collect and extend known finding about the compactness of embedding operators of type \eqref{i-0}, and to study the nuclearity of embeddings of type \eqref{i-0}. Grothendieck introduced the concept of nuclearity in \cite{grothendieck} a long time ago. It can be seen as the basis for many famous developments in functional analysis afterwards, like the influential paper of Enflo \cite{enflo} solving the approximation problem, a long-standing problem of Banach from the Scottish Book. We refer to \cite{Pie-snumb,Pie-op-2}, and, in particular, to \cite{pie-history} for further historic details.

Let $X,Y$ be Banach spaces, $T\in \mathcal{L}(X,Y)$ a linear and bounded operator. Then $T$ is called {\em nuclear}, denoted by $T\in\mathcal{N}(X,Y)$, 
if there exist elements $a_j\in X'$, the dual space of $X$, and $y_j\in Y$, $j\in\mathbb{N}$, such that $\sum_{j=1}^\infty \|a_j\|_{X'} \|y_j\|_Y < \infty$ 
and a nuclear representation $Tx=\sum_{j=1}^\infty a_j(x) y_j$ for any $x\in X$. Together with the {\em nuclear norm}
\[
 \nn{T}=\inf\Big\{ \sum_{j=1}^\infty   \|a_j\|_{X'} \|y_j\|_Y:\ T =\sum_{j=1}^\infty a_j(\cdot) y_j\Big\},
  \]
  where the infimum is taken over all nuclear representations of $T$, the space $\mathcal{N}(X,Y)$ becomes a Banach space. It is obvious that 
	nuclear operators are, in particular, compact.
	
  Already in the early years there was a strong interest to study examples of nuclear operators beyond diagonal operators in $\ell_p$ 
	sequence spaces, where a complete answer was obtained in \cite{tong}.
	Concentrating on embedding operators in spaces of Sobolev type, first results can be found, for instance, in \cite{PiTri,Pie-r-nuc}. We noticed an increased interest in studies of nuclearity in the last years.  Dealing with the Sobolev embedding for spaces on a bounded domain, some of the recent papers we 
have in mind are \cite{EL-4,EGL-3,Tri-nuclear,CoDoKu,CoEdKu,HaLeSk,HaSk-nuc-morrey,HaSkTri-nuc} using quite different techniques however. 

There might be several reasons for this. For example, the problem to describe a compact operator outside the Hilbert space setting is a partly
 open and very important one. 
 It is well known from the remarkable Enflo result \cite{enflo} that there are compact operators
 between Banach spaces which cannot be approximated by finite-rank operators.
 This led to a number of -- meanwhile well-established  and famous -- methods to circumvent this difficulty and find alternative ways to `measure' the compactness or `degree' of compactness of an operator. It can be described, for instance, by the asymptotic behaviour of its approximation and entropy numbers, 
which are basic tools for many different problems nowadays like the eigenvalue distribution of compact operators in Banach spaces, 
optimal approximation of Sobolev-type embeddings, but also for numerical questions. In all these problems, the decomposition
of a given compact operator into a series is an essential proof technique. It turns out that in many of the recent contributions studying nuclearity as mentioned above, a key tool in the arguments are new decomposition techniques as well, adapted to the different spaces. Inspired by the nice paper \cite{CoDoKu} we also used such arguments in our papers \cite{HaSk-nuc-weight,HaLeSk,HaSk-nuc-morrey}, and intend to follow this strategy here again.

It is well-known that function spaces of Besov or Triebel-Lizorkin type, also in the setting of generalised smoothness, defined on $\Rn$ never admit a compact, let alone nuclear embedding. But replacing $\Rn$ by a bounded Lipschitz domain $\Omega\subset\Rn$, then the question of nuclearity in the scale of Besov and Triebel-Lizorkin spaces has already been solved, cf. \cite{Pie-r-nuc} (with a forerunner in \cite{PiTri}) for the sufficient conditions, and 
 \cite{Tri-nuclear} with some forerunner in \cite{Pie-r-nuc} and partial results in \cite{EGL-3,EL-4} for the necessity of the conditions.  More precisely, for Besov spaces on bounded Lipschitz domains, $B^s_{p,q}(\Omega)$, it is well known that
\[
\id_\Omega^\ast : \be(\Omega) \hookrightarrow \bz(\Omega)\] 
\text{is nuclear if, and only if,} 
\[    s_1-s_2 > \nd-\nd \max\left(\frac{1}{p_2}-\frac{1}{p_1},0\right),
\]
where $1\leq p_i,q_i\leq \infty$, $s_i\in\real$, $i=1,2$. The counterpart for spaces of type $B^\sigma_{p,q}(\Omega)$, reads now as follows, see Theorem~\ref{theorem-nuclearity} below: let $1\leq p_i,q_i\leq \infty$, $i=1,2$, and  $\sigma=(\sigma_j)_{j\in\N_0}$, $\tau=(\tau_j)_{j\in\N_0}$ be admissible sequences. Then the embedding 
  \begin{equation*}
  \id_\Omega: \Bsigmae(\Omega)\hookrightarrow \Bsigmaz(\Omega)
  \end{equation*}
 is nuclear if, and only if,
  \begin{equation*}
     \left( \sigma_j^{-1} \tau_j \, 2^{j\nd(\frac{1}{p_1}-\frac{1}{p_2})} 2^{j\nd{\frac{1}{\tn(p_1,p_2)}}}\right)_{j\in\no} \in \ell_{\tn(q_1,q_2)},
  \end{equation*}
    where for $\tn(q_1,q_2)=\infty$ the space $\ell_\infty$ has to be replaced by $c_0$. Here the number $\tn(r_1,r_2)$ is defined via
$$ \frac{1}{\tn(r_1,r_2)} = \begin{cases}
    1, & \text{if} \ 1\leq r_2\leq r_1\leq \infty, \\
    1-\frac{1}{r_1}+\frac{1}{r_2}, & \text{if} \ 1\leq r_1\leq r_2\leq \infty.
  \end{cases}
  $$
  Clearly the two above-mentioned results for $\id_\Omega^\ast$ and $\id_\Omega$ coincide in case of $\sigma_j=2^{j s_1}$, $\tau_j=2^{js_2}$. We obtained parallel results in the context of spaces $F^\sigma_{p,q}(\Omega)$, see Corollary~\ref{Cor-nuc-F} below.

Aiming finally at nuclearity results, we reviewed the many existing partial results about compactness of the operator $\id_\Omega$ which exist in the literature. Applying the same technique as later for the nuclearity outcome, we unified many of the partial results in Theorem~\ref{thm-comp} below: $\id_\Omega$ is compact if, and only if, 
\[
			\left( \sigma_j^{-1} \tau_j\, 2^{j\nd(\frac{1}{p_1}-\frac{1}{p_2})} 2^{j\nd{\frac{1}{p*}}}\right)_{j\in\no} \in \ell_{q*},
\]
where the space $\ell_\infty$ has to be replaced by $c_0$  if $q^*=\infty$. Here $p^\ast$ and $q^\ast$ are the numbers given by
\[\frac{1}{p^\ast}=\max\left(\frac{1}{p_2}-\frac{1}{p_1}, 0\right), \qquad 
  \frac{1}{q^\ast}=\max\left(\frac{1}{q_2}-\frac{1}{q_1}, 0\right),\]
where we may admit $0<p_i, q_i\leq \infty$, $i=1,2$, here. We also obtain a result on the asymptotic behaviour of the entropy numbers $e_k(\id_\Omega)$, see Theorem~\ref{thm-comp} below. It is interesting to note, that parallel to some earlier findings, we observe the same phenomenon again: the criterion for compactness and nuclearity become literally the same, when $p^\ast$ is replaced by $\tn(p_1,p_2)$, and $q^\ast$ by $\tn(q_1,q_2)$.	

  As already indicated, we follow here the general ideas presented in \cite{CoDoKu} which use decomposition techniques and benefit from  Tong's result \cite{tong} about nuclear diagonal operators acting in sequence spaces of type $\ell_p$. This has been extended in \cite{HaLeSk} to vector-valued sequence spaces which paves the way for our present argument.

    The paper is organised as follows. In Section~\ref{sec-spaces} we introduce the concept of spaces of generalised smoothness, recall essential findings, and discuss several approaches which one can find in the literature. In Section~\ref{sec-compact} we present criteria for the embedding of type \eqref{i-0} to be compact, cf. Theorem~\ref{thm-comp}. This generalises a number of special results which we recall below. In particular, we can also provide some general results concerning the asymptotic behaviour of entropy numbers. We discuss some special limiting situations not yet covered by that result, in contrast to different special settings. Our main result about the nuclearity of $\id_\Omega$  can be found in Section~\ref{sec-nuc}.
      }

\section{Function spaces of generalised smoothness -- basic properties and different approaches} \label{sec-spaces}
First we fix some notation. By $\N$ we denote the \emph{set of natural numbers}, by $\N_0$ the set $\N \cup \{0\}$. For $a\in\real$, let   
$a_+:=\max\{a,0\}$.
All unimportant positive constants will be denoted by  $c$, occasionally with subscripts. By the notation 
$$
a_k \sim b_k \quad \mbox{or} \quad \varphi(x) \sim \psi(x)
$$
we always  mean that there are two positive numbers $c_1$ and $c_2$
such that
$$
c_1\,a_k \leq b_k \leq c_2\, a_k \quad \mbox{or} \quad
c_1\,\varphi(x) \leq \psi(x) \leq c_2\,\varphi(x)
$$
for all admitted values of the discrete variable $k$ or the
continuous variable $x$, where $(a_k)_k$, $(b_k)_k$ are
non-negative sequences and $\varphi$, $\psi$ are non-negative
functions. Such sequences or functions are called to be {\it equivalent}.
The symbol $\hookrightarrow$ is used for a continuous embedding from one space into another.




\subsection{Admissible sequences}

We will define function spaces of generalised smoothness based on  the Fourier-analytic characterisation of function
spaces that use a suitable resolution of unity on the Fourier side and a suitable weighted summation of the resulting parts. We will follow the approach from  \cite{FaLe}. To control the support of elements of the resolution of unity and describe the smoothness weights we will use some  sequences, therefore  we start with a definition of  sequences that will be used as  parameters.

 \begin{definition}
 A sequence  $\sigma =(\sigma _j)_{j\in\No}$, with $\sigma _j>0$,
 is called \emph{ an admissible sequence} if there are two constants
 $0<d_0 =d_0(\sigma )\le d_1=d_1(\sigma )<\infty$ such that
\begin{equation}
\label{sigma}
                  d_0\, \sigma _j\le \sigma _{j+1}\le d_1\sigma _j
   \quad \mbox{for any}\quad j\in\No .
\end{equation}
If $d_0>1$, then the sequence will be  called \emph{strongly increasing admissible sequence}.  
  \end{definition}


  \begin{Exams}\label{Ex-2.2}
To illustrate the flexibility of \eqref{sigma} we refer the reader to
some examples discussed in \cite{FaLe} or \cite[Chap. 1]{Bri-tese}.
We give here only the standard example $\sigma^{(1)}=(\sigma_j^{(1)})_{j\in\nat}$ with 
$$ \sigma^{(1)}_j = 2^{sj} \log(1+j)^b, \quad j\in\nat,$$
where $s,b\in\real$. A second, more elaborate example can be found in \cite[Example 3]{Leo-Sk-3}. Let $ 0 \le s_0<s_1$.  
Consider the subsequence $j_{l} = 2^l$ 
and a sequence $(\sigma^{(2)}_j)_{j\in\No}$ defined by
\begin{align}\label{ex-2}
  \begin{cases}
   \ds ~ \sigma^{(2)}_{j_{2l}} ~ := 2^{j_{2l} \frac{2s_1+s_0}{3}}, & \text{and}\quad \sigma^{(2)}_j :=\sigma_{j_{2l}} \,2^{(j-j_{2l})s_0} \qquad~~\text{if}\quad j_{2l} \le j < j_{2l+1}, \\[1ex]
    \sigma^{(2)}_{j_{2l+1}} := 2^{j_{2l+1} \frac{s_1+2s_0}{3}}, & \text{and}\quad
\sigma^{(2)}_j := \sigma_{j_{2l+1}} \,2^{(j-j_{2l+1})s_1}\quad\text{if}\quad j_{2l+1} \le j < j_{2l+2}. \end{cases}
\end{align}
\end{Exams}

To describe  properties of admissible sequences Bricchi and Moura introduced in \cite{BrMo03} the following notion of  Boyd indices of  admissible sequences. 

\begin{definition}
Let  $\sigma =(\sigma _j)_{j\in\No}$ be an admissible sequence and let  
$$
  \overline{\sigma}_j:=  \sup\limits _{k\ge 0}\frac{\sigma _{j+k}}{\sigma
_k}
\quad\mbox{and}\quad
 \underline{\sigma}_j:= \inf\limits _{k\ge 0}\frac{\sigma _{j+k}}{\sigma _k}, \quad j\in\No.
$$
Then the expressions 
\begin{equation}
\label{Boyd} 
  \alpha _{\sigma}:=\inf \limits _{j\in\N} \frac{\log_2  \overline{\sigma}_j}{j} = \lim_{j \to \infty } \frac{\log_2  \overline{\sigma}_j}{j}
\quad\mbox{and}\quad
   \beta _{\sigma}:=\sup \limits _{j\in\N} \frac{\log_2  \underline{\sigma}_j}{j} = \lim _{j\to \infty} \frac{\log_2  \underline{\sigma}_j}{j}
\end{equation}
define, respectively, the {\em upper} 
and  the {\em lower} 
{\em Boyd index  of the sequence} $\sigma $.
\end{definition}

\begin{remark} \label{estimations sigma}
  Plainly,
  $$ 
\log_2 d_0 \le \beta_\sigma \le \alpha_\sigma \le \log_2 d_1,
$$
and for each $\varepsilon > 0$ there exist constants
$c_{0,\varepsilon}
> 0 $ and $c_{1,\varepsilon}  > 0  $ such that for all $j\in\nat$,
\begin{equation}\label{alpha-beta1}
c_{0,\varepsilon}\; 2^{(\beta_\sigma - \varepsilon)j} \le \underline{\sigma}_j \le \sigma_j  \le \overline{\sigma}_j  \le 
c_{1,\varepsilon}\; 2^{(\alpha_\sigma + \varepsilon)j} \;\;.\end{equation}  

 Moreover,  it is easy to see that the Boyd indices of an admissible sequence
    $\sigma$ remain unchanged
    when replacing $\sigma$ by an equivalent
    sequence.
\end{remark}

\begin{Exams} We return to Examples~\ref{Ex-2.2}.
  For $\sigma^{(1)}= (\sigma^{(1)}_j)_j$ with  
  $\sigma^{(1)}_j = 2^{sj} \log(1+j)^b$, $s,b\in\real$, one obviously obtains  $\alpha_{\sigma^{(1)}} = \beta_{\sigma^{(1)}} = s$. But in the second example \eqref{ex-2}, we find that  
$$ \overline{\sigma}^{(2)}_j=  \sup\limits _{k\ge 0}\frac{\sigma^{(2)} _{j+k}}{\sigma^{(2)}
_k} = 2^{js_1} \quad \text{and}\quad
\underline{\sigma}^{(2)}_j= \inf\limits _{k\ge 0}\frac{\sigma^{(2)} _{j+k}}{\sigma^{(2)} _k}= 2^{js_0},$$
consequently $\alpha_{\sigma^{(2)}} = s_1 > \beta_{\sigma^{(2)}} = s_0$ .
\end{Exams}

\subsection{Function spaces of generalised smoothness}
Now we define the function spaces we will work with in the sequel.  In what follows let 
${\cal S}(\Rn)$ denote the Schwartz space of all complex-valued
rapidly decreasing
infinitely differentiable functions on $\Rn$ equipped with the usual
topology and ${\cal S}'(\Rn)$ be its topological dual, the space
of all tempered distributions on $\Rn$. Moreover, let ${\cal F}$ and ${\cal F}^{-1}$ stand respectively for the Fourier transform and its inverse. We choose two admissible sequences   $\sigma=(\sigma_j)_{j\in\N_0}$ and   $N=(N_j)_{j\in\N_0}$ such that 
\begin{align}\label{sigma2}
	    d_0\, & \sigma _j\le \sigma _{j+1}\le d_1\sigma _j  \quad
	    \text{for}\quad d_1\geq d_0>0\quad \text{and any}\quad j\in \No,\\
\intertext{and} \label{N}
\lambda _{0}\, & N_j \le N_{j+1}\le \lambda _1 N_j \quad \text{for}\quad \lambda_1\geq \lambda_0>1\quad \text{and any}\quad
j\in \No,
\end{align} 
i.e., we assume the sequence $N$ not only to be admissible, but also strongly increasing.

Now we can define the suitable resolution of unity.   
 \begin{definition}
 \label{def-decomp}
 For a fixed admissible sequence
 $N=(N_j)_{j\in\N_0}$ satisfying assumption \eqref{N}
 let $\Phi ^{N}$ be the collection of all
 function systems $\varphi^{N}=(\varphi_j^{N})_{j\in\No}$
 satisfying the following conditions:
 \benu[\bfseries\upshape (i)]
\item
   $\varphi_j^{N}\in C_0^\infty(\Rn) \quad\mbox{and}\quad
                  \varphi _j^{N}(\xi )\ge 0 \quad\mbox{if}\quad \xi\in\Rn,
                 \quad\mbox{for any}\quad j\in\No$;
\item
  there exists a number $\kappa_0\in\nat$ with $2 \le
\lambda_0^{\kappa_0}$ such that
  \[
    \begin{cases}
 \supp \varphi _j^{N} \subset \{\xi\in\Rn \, :\, |\xi |\le  N_{j+\kappa_0}\}, & j = 0, 1, \dots, \kappa_{0}-1,\\[1ex] 
 \supp \varphi _j^{N} \subset \{\xi\in\Rn \, :\, N_{j-\kappa_0}\le |\xi |\le N_{j+\kappa_0}\}, & j \ge \kappa_0;
 \end{cases} \]
\item for any $\gamma\in\N_0^\nd$ there exists a constant
$c_{\gamma}>0$
        such that for any $j\in\N_0$
       \begin{equation*}
           \label{fl-i3}
          | \Dd^{\gamma}\varphi _j^{N}(\xi )|
 \le c_{\gamma}\, 
 (1+|\xi|^2)^{-|\gamma |/2}
         \quad\mbox{for any}\quad \xi\in\Rn ;
        \end{equation*}
\item there exists a constant $c_{\varphi}>0$ such that
 \begin{equation*}
 \label{fl-i4}
            0< \sum\limits _{j=0}^{\infty}\varphi _j^{N}(\xi )
 =c_{\varphi}<\infty
             \quad\mbox{for any}\quad \xi\in\Rn   .
 \end{equation*}
\eenu
\end{definition}

 \begin{remark}
 	If $N_j=2^j$, $j\in \N_0$, then the function system $\varphi^{N}=(\varphi_j^{N})_{j\in\No}$ is a classical smooth dyadic resolution of unity. In that situation we will omit the superscript $N$ and will write  $(\varphi_j)_{j\in\No}$ instead of $(\varphi_j^{N})_{j\in\No}$. 
 \end{remark}

\begin{definition} \label{defbf2}
	Let $\sigma=(\sigma_j)_{j\in\N_0}$ and $N=(N_j)_{j\in\N_0}$ be admissible
	sequences and assume that $N=(N_j)_{j\in\N_0}$ satisfies \eqref{N}. Moreover, let $\varphi ^{N}\in\Phi ^{N}$.
  \benu[\bfseries\upshape (i)]
\item Let $0<p\le \infty$ and $0 <q\le\infty$.
 The {\em Besov space of generalised smoothness} $B^{\sigma ,N}_{p,q}(\Rn)$  is defined as
 $$
 \Big\{
 f\in {\cal S}'(\Rn)\,:\,
  \Big\|f\,|\, B^{\sigma ,N}_{p,q}(\Rn) \Big\| :=
    \Big(\sum\limits _{j=0}^{\infty}
            \sigma_j^q \, \|{\cal F}^{-1}\, (\varphi_j^N {\cal F}f)|L_p(\Rn)\|^q\,
    \Big)^{1/q} <\infty\Big\},
 $$
with the usual modification when $q=\infty$.
\item
             Let $0<p< \infty$ and $0 < q\le \infty$.
 The {\em Triebel-Lizorkin space  of generalised smoothness} $F^{\sigma , N}_{p,q}(\Rn)$ is defined as
 $$
 \Big\{
 f\in {\cal S}'(\Rn)\,:\,
 \|f\,|\, F^{\sigma ,N}_{p,q}(\Rn)\| :=
    \Big\|
    \Big(\sum\limits _{j=0}^{\infty}
            \sigma_j^q \, |{\cal F}^{-1}\, (\varphi_j^N {\cal F}f)(\cdot)|^q\,
    \Big)^{1/q}
   |\, L_p(\Rn) \Big\|
 <\infty\Big\},
 $$
with the usual modification when $q=\infty$.
\eenu
\end{definition}

\begin{remark}
  Both  $B^{\sigma ,N}_{p,q}(\Rn)$ and $F^{\sigma , N}_{p,q}(\Rn)$
	are (quasi)-Banach spaces, Banach if $\min(p,q)\ge 1$, which are independent of the choice of the
	system $(\varphi_j ^{N} )_{j\in \N_0}$ and with respect to equivalent sequences, in the sense of equivalent quasi-norms. As in the classical case,
	the embeddings
	${\cal S}(\Rn) \hookrightarrow B^{\sigma , N}_{p,q}(\Rn)\hookrightarrow{\cal S}'(\Rn)$ and
	${\cal S}(\Rn) \hookrightarrow F^{\sigma ,N}_{p,q}(\Rn)  \hookrightarrow{\cal S}'(\Rn)$ hold true for all admissible
	values of the parameters and sequences.
	If $ p,q<\infty$, then ${\cal S}(\Rn)$
	is dense in $B^{\sigma , N}_{p,q}(\Rn)$ and in $F^{\sigma ,N}_{p,q}(\Rn)$.
	Moreover, it is clear that $B^{\sigma ,N}_{p,p}(\Rn)=F^{\sigma ,N}_{p,p}(\Rn)$.
        
Note also that if $N=(2^j)_{j\in\N_0}$ and $\sigma = \sigma ^s := (2^{js})_{j\in\No}$
	with $s$ real,
	then
	the above spaces coincide with the usual function spaces $B^s_{p,q}(\Rn)$
	and $F^s_{p,q}(\Rn)$ on $\Rn$,  respectively.
We shall use the simpler notation $B^s_{p,q}(\Rn)$ and $F^s_{p,q}(\Rn)$ in the more classical situation just mentioned.
Even for general admissible $\sigma$, when  $N=(2^j)_{j\in\N_0}$ we shall write simply $F^\sigma_{p,q}(\Rn)$ and $B^\sigma_{p,q}(\Rn)$ instead of $F^{\sigma ,N}_{p,q}(\Rn)$ and $B^{\sigma ,N}_{p,q}(\Rn)$, respectively.

We refer to \cite{FaLe} for further characterisations of the spaces, in particular, characterisation by local means and atomic decompositions. There one can also find some discussion concerning the properties of the sequences we use in Definition~\ref{defbf2}, in particular, the importance of the assumption \eqref{N}. 	
\end{remark}

We have the following useful relation between generalised $B$ and $F$ spaces in analogy to the classical elementary embeddings. 

\begin{proposition}\label{BF}
Let $0<p< \infty$, $0 < q \leq \infty$. Let $N$ and $\sigma$ be admissible sequences with $N$ satisfying \eqref{N}. Then
\begin{equation} \label{embBFB}
B^{\sigma,N}_{p,\min\{ p,q \}}(\Rn) \hookrightarrow F^{\sigma,N}_{p,q}(\Rn) \hookrightarrow B^{\sigma,N}_{p,\max\{ p,q \}}(\Rn).
\end{equation}
\end{proposition}
The proposition can be proved  similarly to the classical case, cf. \cite[Prop. 2.3.2/2(iii), p.47]{T-F1}. For further embeddings results we refer also to \cite{CaFa04,CH04,CaLe1,HM08}.

In \cite{CaLe1} Caetano and Leopold  used the so-called {\em standardization procedure}, which allows  for a given pair of admissible sequences $\sigma$ and $N$ to define a new admissible sequence $\beta$, such that $\beta$ together with the standard sequence  $(2^j)_{j\in\N_0}$ defines the same space with generalised smoothness.  More precisely, let $N$ and $\sigma$ be admissible sequences and let $N$ satisfy  the condition \eqref{N}. We choose  $\kappa_0\in \N$ such that  
$ \lambda_0^{\kappa_0} \ge 2$~.
Define
\begin{equation} \label{beta}
	\beta_j := \sigma_{k(j)}, \quad \mbox{with }\; k(j):= \min \{ k \in \No : 2^{j-1} \leq N_{k+\kappa_0} \}, \;\; j \in \No .
\end{equation}
Then we have that
\[
\mu_0 \beta_j \leq \beta_{j+1} \leq \mu_1 \beta_j, \quad j \in \No,
\]
with $\mu_0 = \min \{ 1, d_0^{\kappa_0} \}$, $\mu_1 = \max \{ 1, d_1^{\kappa_0} \}$.

Under these conditions it was proved
in \cite[Theorem 1]{CaLe1} the following result.

\begin{theorem} \label{standard}
	Let $N$ and $\sigma$ be admissible sequences and let  $N$ satisfy  \eqref{N}.
	Let further $0 < p,q \leq \infty$ (with $p < \infty$ in the $F$-case). Then
	\[
	F^{\sigma,N}_{p,q}(\Rn) = F^\beta_{p,q}(\Rn) \qquad \text{and} \qquad  B^{\sigma,N}_{p,q}(\Rn) = B^\beta_{p,q}(\Rn),
	\]
	where $\beta := ({\beta}_{j})_{j\in\No}$ is determined by \eqref{beta}.
\end{theorem}

{\em Convention}.\quad 
Let $\sigma$ be an admissible sequence. We adopt the nowadays usual custom to write $A^\sigma_{p,q}(\Rn)$ instead of $B^\sigma_{p,q}(\Rn)$ or $F^\sigma_{p,q}(\Rn)$ when both scales of spaces are meant simultaneously in some context.\\

Let us finally mention that generalised spaces on open sets in $\Rn$ can be defined in the standard way by restriction. We recall this definition  for completeness.  Let $\Omega$ be an open set  in $\Rn$,  $\Omega\not= \Rn$. 
If $\sigma$ is an admissible sequence, then 
\begin{equation}\label{A-Omega}
A^\sigma_{p,q}(\Omega) := \{f\in \cD(\Omega): \ \exists\ g\in A^\sigma_{p,q}(\Rn)\quad \text{with}\ g|_\Omega = f\}. 
\end{equation}
Equipped with the quotient norm 
\begin{equation}\label{A-Omega-norm}
\| f| A^\sigma_{p,q}(\Omega)\| := \inf\left\{ \|g|A^\sigma_{p,q}(\rn)\|: \ g\in\A(\rn)\quad \text{with}\ g|_\Omega = f\right\},
\end{equation}
$A^\sigma_{p,q}(\Omega) $ becomes a (quasi-)Banach space. 


\subsection{Alternative definitions} 
Sometimes it is more convenient to use   spaces of generalised smoothness defined by  function parameters instead of admissible sequences. In this section we first give an alternative description of the spaces introduced in Definition~\ref{defbf2}, but now in terms of function parameters. Afterwards we recall the corresponding definitions given  by Merucci and by Edmunds and Netrusov. We shall prove that they are covered by our Definition~\ref{def-decomp}.      

 \subsubsection{ Approach via function parameters} 
We define  a class of suitable functions, corresponding to admissible sequences, and describe the spaces of generalised smoothness in terms of these functions. 

 \begin{definition} \label{weight1} 
Let   $\varphi : [1,\infty) \to (0,\infty) $ be a measurable function. We say that $\varphi$  belongs to a family $\mathcal{V}$ if 
\begin{align*} 
	\underline{\varphi}(t) & := \inf_{s\in [1,\infty)} \frac{\varphi(ts)}{\varphi(s)} \ > 0 \quad ~\text{ for all }~ t \in [1,\infty),  ~\nonumber \\ 
	\nonumber \overline{\varphi}(t) &  := \sup_{s\in [1,\infty)} \frac{\varphi(ts)}{\varphi(s)} \ 
< \infty \quad\text{ for all }~ t \in [1,\infty),
\end{align*}
and both functions  $\underline{\varphi}$ and $\overline{\varphi}$ are measurable.
\end{definition}

\begin{remark}
The above definition was given in \cite{K-L-S-S-3} in the context of weighted Besov spaces. However conditions of the above type  have been used earlier  e.g. in connection with
real interpolation with a function parameter or with generalised smoothness. We refer to \cite{K-L-S-S-3} for details,  examples and further references.
\end{remark}
If $\varphi\in \mathcal{V}$, then  for all $ s \ge 1 $ and $t \ge 1 $ we have obviously 
\begin{equation} \nonumber \underline{\varphi}(t) \, \varphi(s) \le \varphi(ts) 
	\le \overline{\varphi}(t)\varphi(s) \, .
\end{equation}

The functions $ \overline{\varphi}(t) $ and  $1/\underline{\varphi}(t)$  are submultiplicative. This guarantees that 
$$\alpha_{\varphi}= \inf_{t>1} \frac{\log_2\overline{\varphi}(t)}{\log_2 t}   
$$ 
is either a real number or equals $-\infty$ and 
$$\beta_{\varphi}= \sup_{t>1} \frac{\log_2\underline{\varphi}(t)}{\log_2 t}   
$$ 
is either a real number or equal to $+\infty$.  As stated in \cite[Lemma~4.8]{K-L-S-S-3}, for $\varphi\in \mathcal{V}$ it turns out that $-\infty<\beta_{\varphi}\le \alpha_{\varphi} <\infty$ and we have the following counterpart of \eqref{alpha-beta1}: for any $\varepsilon >0$ there is a constant 
$c_{\varepsilon}\ge 1$ such that 
\begin{equation} \label{varphi}
	c^{-1}_\varepsilon \,  s^{\beta_{{\varphi}}-\varepsilon} \le \underline{\varphi}(s) 
	\le 
	\frac{\varphi (s)}{\varphi (1)} \le 
	\overline{\varphi}(s) \le c_\varepsilon \, s^{\alpha_{{\varphi}}+\varepsilon}
	\qquad \mbox{for all}\quad s \ge 1.  
\end{equation}
We call $\beta_{\varphi}$ and $\alpha_{\varphi}$ the {\em upper} 
and  the {\em lower} 
{\em Boyd index  of the  function $\varphi\in \mathcal{V}$, respectively.   }

\begin{remark}
The following relations between functions belonging to the class $\mathcal{V}$ and admissible sequences were noticed in \cite{CaLe1}. 
\begin{itemize}
	\item For each function $\Sigma \in \mathcal{V} $ the sequence $\sigma=(\sigma_k)_{k\in\N_0}$ with $\sigma_k := \Sigma(2^k) $ is admissible, with $d_0 = \underline{\Sigma}(2)$ and $d_1 = \overline{\Sigma}(2)$, cf. \eqref{sigma}.
        Moreover if ${\mathcal N}\in {\mathcal V}$ is strictly increasing 	and there exists $\lambda_0>1$ such that $\lambda_0 \mathcal{N}(t)< \mathcal{N}(2t)$ for any $t\ge 1$, then the sequence $N=(N_k)_{k\in\N_0}$  with $N_k= \mathcal{N}(2^k)$ is admissible and satisfies \eqref{N}.           
	\item Vice versa, for each admissible sequence $\sigma=(\sigma_k)_{k\in\mathbb{N}_0}$ there exists a {corresponding function} $\Sigma \in \mathcal{V}$, that is, a function such that  $\Sigma (t) \sim \sigma_k$ ~for all ~$t\in[2^k,2^{k+1})$. For example, we can define 
	$$ \Sigma(t) = \sigma_k + (2^{-k}t - 1)(\sigma_{k+1}-\sigma_k) \quad\text{ for }\quad t\in[2^k,2^{k+1}).$$
	 If, in addition, an admissible sequence $(N_k)_{k\in\mathbb{N}_0}$ fulfils the condition \eqref{N},    then there exists a corresponding function $\mathcal{N}\in \mathcal{V}$, strictly increasing and with $\lambda_0 \mathcal{N}(t) \le \mathcal{N}(2t) $ for all $ t \ge 1$. The function $\mathcal{N}$ can be defined by the same formulae as above.  
\end{itemize}

 Using the above correspondence one can define spaces $F^{\sigma,N}_{p,q}(\Rn)$ and  $B^{\sigma,N}_{p,q}(\Rn)$ starting from functions $\Sigma, \mathcal{N}\in \mathcal{V}$, $\mathcal{N}$ with properties as described above.  We can just take the admissible sequences $\sigma=(\sigma_k)_{k\in\N_0}$ and $N=(N_k)_{k\in\N_0}$ with
 $\sigma_k = \Sigma(2^k)$ and $N_k=\mathcal{N}(2^k)$ in Definition~\ref{defbf2}.   
So both starting points, the sequential one and the functional one, lead us to the same family of generalised spaces. Moreover, one can even choose functions and sequences in such a way that the corresponding  Boyd indices coincide. In some abuse of notation we may occasionally write $\sigma (2^k)$ instead of $\sigma_k$, for a given admissible sequence $\sigma=(\sigma_k)_{k\in\N_0}$, thus denoting  also by $\sigma$ a function in $\mathcal{V}$ corresponding to the admissible sequence $\sigma$.

  	One can also formulate Theorem~\ref{standard} in terms of functional parameters.  Let $\Sigma$ and $\mathcal{N}$ be  corresponding functions belonging to $\mathcal{V}$ with $\mathcal{N}$ strictly increasing, continuous and such that $\lambda_0\mathcal{N}(t)<\mathcal{N}(2t)$ for any $t\ge 1$.  If $\mathcal{N}^{-1}$  is the inverse function to $\mathcal{N}$, defined on $[\mathcal{N}(1),\infty)$, and 
  	\[ \hat{\beta_j} := \Sigma(\mathcal{N}^{-1}(2^j)), \quad\text{for}\quad j \in \mathbb{N}_0\; \text{ large enough}, \]
  	then $(\hat{\beta_j})_{j}$ is an admissible sequence, equivalent to the sequence $\beta$ defined for the corresponding admissible sequences $\sigma$ and $N$ by formulae \eqref{beta}. 
  	In  consequence $B^{\hat{\beta}}_{p,q}(\Rn) = B^{\beta}_{p,q} (\Rn)=  B^{\sigma,N}_{p,q}(\Rn)$ and  $F^{\hat{\beta}}_{p,q} (\Rn)= F^{\beta}_{p,q} (\Rn)=  F^{\sigma,N}_{p,q}(\Rn)$, respectively. 
  \end{remark}
   

\subsubsection{The approach related to interpolation} 
In connection with real interpolation with a function parameter Merucci introduced Besov spaces of generalised smoothness $B^\phi_{p,q}(\Rn)$, cf. \cite{Me}. Here $\phi:(0,\infty) \to (0,\infty)$ denotes a positive continuous function that satisfies the condition \eqref{B-class} below. These 
spaces  were later used by {Cobos and Fernandez  in \cite{CoFe}}, Cobos and K\"uhn in \cite{CoKu}, and many other authors. 

\begin{definition} A continuous function   
$\phi: (0,\infty) \to (0,\infty)$ with $\phi(1)=1$  belongs to the class $\mathcal{B}$ if 
\begin{equation}\label{B-class}
\overline{\phi}(t) = \sup_{u>0} {\frac{\phi(tu)}{\phi(u)}} < \infty  \quad\text{ for every }~~ t > 0.
\end{equation}
\end{definition}

For $\phi\in\mathcal{B}$ and $0<p,q\le\infty$, the Besov space of generalised smoothness $B^\phi_{p,q}(\Rn)$ consists of
all $f \in \mathcal{S}'(\Rn )$ having a finite quasi-norm
$$ 
\Big\|f\,|\, B^{\phi}_{p,q}(\Rn) \Big\| :=
\left(\sum\limits _{j=0}^{\infty}
\phi(2^{j})^q \, \|{\cal F}^{-1}\, (\varphi_j {\cal F}f)|L_p(\Rn)\|^q\,
\right)^{1/q}, 
$$
with the usual modification when $q=\infty$.

The classes   $\mathcal{B}$ and  $\mathcal{V}$ define the same spaces. First of all, due to Theorem~\ref{standard}, we can restrict our attention to the spaces defined by the classical dyadic resolution of unity. 
For each function $\phi$ belonging to $\mathcal{B}$ we can define  a function  {$\varphi$} belonging to $\mathcal{V}$ by restriction. It results in $\phi(2^j) = \varphi(2^j)$, $j\in\no$, and 
$\overline{\varphi}(t) \le \overline{\phi}(t) < \infty$, $t\geq 1$, is obvious. On the other hand, we have also
$$\underline{\varphi}(t) = \inf_{s\ge 1} {\frac{\varphi(ts)}{\varphi(s)}} = \frac{1}{\sup_{s \ge 1} {\frac{\varphi(s)}{\varphi(ts)}}} = \frac{1}{\sup_{s \ge 1} {\frac{\varphi((ts)t^{-1})}{\varphi(ts)}}}\ge \frac{1}{\overline{\phi}(t^{-1})} > 0 ~~.  $$
Conversely, an easy calculation shows that the extension of a function $\varphi \in \mathcal{V} $, normed with $\varphi(1)=1$, by the constant $1$ for $0 < t < 1$ belongs to $\mathcal{B}$. 
Consequently,  the function spaces $B^\sigma_{p,q}(\Rn)$, $B^\varphi_{p,q}(\Rn)$ and $ B^\phi_{p,q}(\Rn)$ with $\phi\in \mathcal{B}$, $\varphi\in \mathcal{V}$ and $\sigma$ being an admissible sequence, coincide, cf. also   \cite[Section 2.2]{CaFa04} and  \cite[Proposition 3]{Al}. 

\begin{remark}
	In \cite{BrMo03} Bricchi and Moura consider a class of functions $\mathbb{B}$ that is very close to the class $\mathcal{B}$.  They add one more condition. Namely, a function $g$ belongs to $\mathbb{B}$ if it  {belongs to $\mathcal{B}$ and additionally satisfies } $g(t^{-1})= g(t)^{-1}$ for any $t>0$. But this change is immaterial in the sense that both classes of function parameters lead to the same function spaces. {The only point is that the Boyd indices of the admissible sequence and of the corresponding function in $\mathbb{B}$ coincide,  cf. \cite[Proposition 3.6]{BrMo03}. The same may not happen in general  for functions in $\mathcal{B}$, in particular when one considers a  function in $\mathcal{B}$ defined as being constant in the interval $(0,1)$, as remarked in   \cite{LN22}. Note that the Boyd indices of a function in $\mathcal{B}$ are defined  similarly to the ones of a function in $\mathcal{V}$ (taking  now in their definition the infimum and the supremum over the interval $(0,\infty)$ instead of the interval $[1,\infty)$).}
\end{remark}

\subsubsection{Another approach via parameter functions}\label{EdNet-approach}
Edmunds and Netrusov used in \cite{EdNe} a bit different notation. They considered parameter  functions $\omega: (0,1] \to (0,\infty) $ that satisfy the following condition:  
there exist  positive constants $L, c>0$ such that  for all $0<t_1\le t_2\le 1$ it holds
\begin{equation} \label{fomega} \omega(t_1) t_1^{-L} \ge c\,  \omega(t_2) t_2^{-L}  \qquad\text{and}\qquad  c\,\omega(t_1) t_1^{L} \le   \omega(t_2) t_2^{L}~.
\end{equation}

For $0<p,q\le\infty$, the Besov space of generalised smoothness $B^\omega_{p,q}(\Rn)$ consists of
all $f \in \mathcal{S}'(\Rn )$ having a finite quasi-norm
$$ 
\Big\|f\,|\, B^{\omega}_{p,q}(\Rn) \Big\| :=
\left(\sum\limits _{j=0}^{\infty}
 \Big(\frac{1}{\omega(2^{-j})}\Big)^q \, \|{\cal F}^{-1}\, (\varphi_j {\cal F}f)|L_p(\Rn)\|^q\,
\right)^{1/q}, 
$$
with the usual modification when $q=\infty$. 
Again the above function spaces are covered by Definition~\ref{defbf2}.   Let $\varphi \in \mathcal{V}$ and define $\omega(t) := 1/\varphi(t^{-1})$, then $\omega $ fulfils the condition \eqref{fomega} with {$L=\max\{\alpha_{\varphi},- \beta_{\varphi}\}+\varepsilon$} 
and  {$c=c_\varepsilon^{-1}$, cf. } \eqref{varphi}. 
On the other hand, taking  $\sigma_j := 1/\omega(2^{-j})$ we get an admissible sequence with $d_0 = c\, {2^{-L}}$ and $d_1 = c^{-1} 2^L $ which again defines the same function space.

\begin{remark}For the sake of completeness we remark that function spaces of generalised smoothness also appear as trace spaces for function spaces on fractals, so-called $(\alpha,\Psi)$- or $h$-sets. We do not want to explicate the details here, but refer to \cite{Bri-tese,Mo2}.
\end{remark}


\subsection{Wavelet decomposition}\label{wavdecomp}
We briefly recall the wavelet characterisation of Besov spaces with general smoothness  proved in \cite{Al}. It will be essential in our approach. 
We will assume in this section that $(\sigma_j)_{j\in\No}$ is an admissible sequence with $\sigma_0 = 1$ and $\Sigma$ is a  corresponding function belonging to $\mathcal{B}$.  

Let $\widetilde{\phi}$ be a scaling function  on $\R$ with compact support and of sufficiently high regularity.
Let $\widetilde{\psi}$ be an associated wavelet. Then the  tensor-product ansatz yields a scaling function $\phi$  and associated wavelets
$\psi_1, \ldots, \psi_{2^{\nd}-1}$, all defined now on $\Rn$.  We suppose $\widetilde{\phi} \in C^{r}(\R)$ and $\supp \widetilde{\phi}
\subset [-\nu_0,\, \nu_0]$ for certain natural numbers $r$ and $\nu_0$. This implies
\begin{equation}\label{2-1-2}
	\phi, \, \psi_\ell \in C^{r}(\Rn) \quad \text{and} \quad 
	\supp \phi ,\, \supp \psi_\ell \subset [-\nu,\, \nu]^\nd , 
\end{equation}
for $\ell=1, \ldots \, , 2^{\nd}-1$. Moreover 
\[\int_{\R^d} x^\alpha \psi_\ell(x) \dint x = 0 \qquad\text{for}\qquad |\alpha|<r .\]
We use the same abbreviations for the wavelet system as in \cite{Al}, that is,
\[ 
\psi^{\ell}_{j,m}(\cdot) := \begin{cases} 
	\phi(\cdot - m), & j=0,\  m\in \Zn,\  {\ell}=1 \\ 
	\psi^{\ell}_{j,m}(2^{j-1} \cdot - m), & j\in\N,\ m\in\Zn, \ 1\le \ell \le L=2^{\nd}-1.  \end{cases} 
\]
%
Let $L_0=1$ and $L_j=L=2^{\nd}-1$ if $j\in \N$.  Then 
 $$\{2^{j\nd/2} \psi^{\ell}_{jm}: \ j\in\mathbb{N}_0,  \ 1\le \ell \le L_j, \ m\in\Zn\}$$ 
is an orthonormal basis in $L_2(\Rn)$. 

We put  $I=\{(\ell,j,m): j\in \No,\ 1\le \ell \le L_j,\ m\in \Zn\}$ and $I'=\{(\ell,j): j\in \No,\ 1 \le \ell \le L\}$.
Let $0<p,q\le \infty$ an let $\sigma=(\sigma_j)_{j\in \N_0}$ be an admissible sequence. We need first to introduce the appropriate sequence
spaces $b^\sigma_{p,q}$ which consist of all complex-valued sequences 
$\lambda= (\lambda^\ell_{j,m})_{(\ell,j,m)\in I}$ such
that the following quasi-norm
$$\|\lambda|b^\sigma_{p,q}\| := \left( \sum_{(\ell,j) \in I'} (\sigma_j 2^{-j\nd/p})^q \left(\sum_{m\in \Zn} |\lambda^{\ell}_{jm}|^p\right)^{q/p}\right)^{1/q} $$
(with the usual modifications if $p=\infty$ and/or $q=\infty$) is finite. 

Now we can formulate the required wavelet decomposition result.

\begin{theorem}[{\cite[Theorem 13]{Al}}] \label{wavelet-thm}
  Let $0 < p \le \infty$, $0<q \le \infty$, 
  and let $\sigma=(\sigma_j)_{j\in\N_0}$ be an admissible sequence with $\sigma_0=1$. Then there exists a number $r(\sigma,p)$ such that, for any $r\in \N$ with $r> r(\sigma,p)$, the following holds.\\
Any given $f\in \mathcal{S}'(\Rn)$ belongs to $B^\sigma_{p,q}(\Rn)$ if, and only if, it can be represented as 
\[
f   =  \sum_{(\ell,j,m) \in I} \lambda^\ell_{j,m} \psi^\ell_{j,m} \quad\mbox{ with }\quad \lambda=(\lambda^\ell_{j,m}) \in  b^\sigma_{p,q}~,
\]
with unconditional convergence in $\mathcal{S}'(\Rn)$. Moreover, the coefficients $\lambda^\ell_{j,m}$ are uniquely
 determined by
\[ \lambda^\ell_{j,m}= \lambda^\ell_{j,m}(f) := 2^{j\nd}\langle f,\psi^\ell_{j,m}\rangle. 
\] 
Furthermore,
\[ \|f|B^\sigma_{p,q}(\Rn)\| \sim \| \lambda(f)| b^\sigma_{p,q} \|\]
(equivalent quasi-norms), where $\lambda(f) = (\lambda^\ell_{j,m}(f))_{(\ell,j,m)\in I}$.   
\end{theorem} 

\begin{corollary} {Under the same conditions of Theorem~\ref{wavelet-thm},} the mapping
	\[\mathcal{I}:f\mapsto \lambda(f)\]
establishes a topological isomorphism from $B^\sigma_{p,q}(\Rn)$ onto $ b^\sigma_{p,q}$. 
\end{corollary}

\begin{remark}
Theorem 13 of \cite{Al} is stated for $p$ finite, however the result is also valid for $p=\infty$, as observed in \cite[p.~628]{EdNe}, see also  \cite[Theorem 4.1]{EdNe}.
\end{remark}

In next sections we shall mainly work with embeddings of spaces defined on a bounded Lipschitz domain $\Omega\subset\Rn$. So we need a topological isomorphism between the function space and the corresponding sequence space not only for  $B^\sigma_{p,q}(\Rn)$, but also for  $B^\sigma_{p,q}(\Omega)$.  Cobos, Domínguez and  K\"{u}hn proved in \cite{CoDoKu}  that such an isomorphism exists. We repeat their argument based on universal extension operators for Besov spaces defined on  Lipschitz domains and the real interpolation with function parameter.

Let $0 < p, q \le \infty $ and let $\sigma=(\sigma_j)_{j\in\N_0}$ be an admissible sequence with $\sigma_0=1$. Recall that $\Sigma$ is assumed to be a function in $\mathcal{B}$ corresponding to the sequence $\sigma$, that is with $\Sigma(2^j)=\sigma_j$. Moreover, it can be assumed that $\sigma$ and $\Sigma$ have the same Boyd indices. We follow the same notation as in  \cite{CoFe} and  consider  $s_1, s_0\in \R$ so that 
\[ s_1 < \beta_{\sigma}  \le \alpha_{\sigma} < s_0.\]
 Then 
 \begin{equation}\label{int-gamma} \left(B^{s_0}_{p,q}(\Rn),B^{s_1}_{p,q}(\Rn)\right)_{\gamma,q} =  B^{\sigma}_{p,q}(\Rn) 
 \end{equation}
 with 
 \[ \gamma(t) := \frac{t^{\frac{s_0}{s_0-s_1}}}{\Sigma(t^{\frac{1}{s_0 - s_1}})},  \]
 cf. \cite[Theorem 5.3, Remark 5.4]{CoFe}.

The restriction operator $\re(g) = g|\Omega $ is always bounded  and linear, both from   $B^{\sigma}_{p,q}(\Rn)$ onto $ B^{\sigma}_{p,q}(\Omega)$ and  from   $B^{s_k}_{p,q}(\Rn)$ onto $ B^{s_k}_{p,q}(\Omega)$, $k=0,1$. But the interpolation formula \eqref{int-gamma} implies 
\[ \re:  B^{\sigma}_{p,q}(\Rn) \longrightarrow    \left(B^{s_0}_{p,q}(\Omega),B^{s_1}_{p,q}(\Omega)\right)_{\gamma,q} ,  \]
so  $B^{\sigma}_{p,q}(\Omega) \hookrightarrow \left(B^{s_0}_{p,q}(\Omega),B^{s_1}_{p,q}(\Omega)\right)_{\gamma,q}$. 
 
If $\Omega\subset\Rn$ is a Lipschitz domain, there exists a universal linear and bounded extension operator $\ext $ for all classical Besov spaces,   see Rychkov \cite{Ryc} or \cite[pp.64-66]{T-F3}. In particular,
\[ \ext: B^{s_k}_{p,q}(\Omega) \longrightarrow  B^{s_k}_{p,q}(\Rn),\quad k=0,1, \]
is bounded. 
Using again the interpolation argument  we get 
\[ \ext:  \left(B^{s_0}_{p,q}(\Omega),B^{s_1}_{p,q}(\Omega)\right)_{\gamma,q} \longrightarrow B^{\sigma}_{p,q}(\Rn) . \]
For each $f \in  \left(B^{s_0}_{p,q}(\Omega), B^{s_1}_{p,q}(\Omega)\right)_{\gamma,q} $ it follows that $\ext(f) \in  B^{\sigma}_{p,q}(\Rn)$ and so $ f = \re(\ext(f)) \in  B^{\sigma}_{p,q}(\Omega)$. Moreover,
\[ \|f| B^{\sigma}_{p,q}(\Omega)\| \le \| \ext(f)| B^{\sigma}_{p,q}(\Rn)\| \le c \, \| f |   \left(B^{s_0}_{p,q}(\Omega),B^{s_1}_{p,q}(\Omega)\right)_{\gamma,q} \| .\]
This yields that $  B^{\sigma}_{p,q}(\Omega) = \left(B^{s_0}_{p,q}(\Omega),B^{s_1}_{p,q}(\Omega)\right)_{\gamma,q}$ with equivalent norms. It also shows that  the linear operator $\ext$ maps  $  B^{\sigma}_{p,q}(\Omega)$ into   $  B^{\sigma}_{p,q}(\Rn)$. 
Now let 
\[I_j = \{(\ell,m): 1 \le \ell \le L_j , m \in \Zn \mbox{~and~} \supp{\psi^{\ell}_{jm}} \cap \Omega \not= \emptyset \} ~~.\]
Since $\Omega$ is bounded with non-empty interior and the wavelets $\psi^{\ell}_{jm}$ are compactly supported we get 
$ M_j = \# I_j \sim 2^{j\nd}$. Using the extension operator $\ext$ it follows from the wavelet characterisation that 
\begin{equation}\label{wavebound}
	   B^{\sigma}_{p,q}(\Omega)\quad \text{is isomorphic to }\quad   \ell_q(\sigma_j 2^{-j\frac{\nd}{p}} \ell_p^{M_j})\quad \text{ with } \quad M_j \sim 2^{j\nd}, \ 0<p \le \infty, \ 0<q\le \infty,
\end{equation}
and $\sigma=(\sigma_j)_{j\in \N_0}$ being an admissible sequence. 
Here $ \ell_q(\beta_j \ell_p^{M_j})$, $0 < p,q \le \infty$, denotes the space of all complex-valued sequences $\lambda = (\lambda_{j,\ell})_{j\in\No,\ell=1,...M_j}$ such that 
$$\|\lambda | \ell_q(\beta_j \ell_p^{M_j})\| := \left(\sum_{j=0}^\infty \beta_j^q\left(\sum_{\ell=1}^{M_j} |\lambda_{j,\ell}|^p \right)^{q/p}\right)^{1/q} < \infty$$
(with the usual modifications if $p=\infty$ and/or $q=\infty$) is finite, where $(\beta_j)_{j\in\No}$ is an admissible sequence with $\beta_j > 0$ and $(M_j)_{j\in\No}$ is a sequence of natural numbers.

\section{Compact embeddings  and their entropy numbers}\label{sec-compact}
In this section we want to prove the necessary and sufficient conditions for the compactness of  embeddings of spaces with generalised smoothness on bounded Lipschitz domains. In addition, we want to estimate the asymptotic behaviour of the entropy numbers of those compact embeddings.  Here we also collect the various partial results already known for a longer time. But first  we briefly recall the concept of entropy numbers.

\begin{definition}\label{defi-ak}
	Let $ X $ and $Y$ be two complex (quasi-) Banach spaces, 
	$k\in\nat\ $ and let $\ T\in\mathcal{L}(X,Y)$ be a linear and 
	continuous operator from $ X $ into $Y$. 
	The {\em k\,th entropy number} $\ e_k(T)\ $ of $\ T\ $ is the
	infimum of all numbers $\ \varepsilon>0\ $ such that there exist $\ 2^{k-1}\ $ balls
	in $\ Y\ $ of radius $\ \varepsilon\ $ which cover the image $\ T(B_X)$ of the closed unit ball $\ B_X:=\{x\in
	X:\;\|x|X\|\leq 1\}$.
\end{definition}

\begin{remark}     	
	For details and properties of entropy numbers we refer to \cite{CS,EE,Koe,Pie-snumb} (restricted to the case of Banach spaces), and \cite{ET} for some extensions to quasi-Banach spaces. 
		Note that  the sequence $(e_k(T))_{k\in \N}$ is non-increasing and $e_1(T)\leq \|T\|$. One can easily show that 	$\ T\ $ {is compact} {if, and only if,} $\ \lim_{k\rightarrow\infty} e_k(T)= 0\ $. So the asymptotic behaviour of entropy numbers i.e., their decay,  measure "how compact" the operator $T$ is. Moreover, the asymptotic behaviour has an application to estimate eigenvalues of compact operators, cf.  \cite{CS,EE,ET,Koe,Pie-snumb} for further details.        
	     Among other properties of entropy numbers we only want to mention 
	the multiplicativity: let $X,Y,Z$
	be complex (quasi-) Banach spaces and $\ T_1 \in\mathcal{L}(X,Y)$, $ T_2 \in\mathcal{L}(Y,Z)$. Then
	\beq
	e_{k_1+k_2-1} (T_2\circ T_1) \leq e_{k_1}(T_1)\,
	e_{k_2} (T_2),\quad k_1, k_2\in\nat.
	\label{e-multi}
	\eeq
	\end{remark}


In the classical setting of Besov and Triebel-Lizorkin spaces defined on  bounded Lipschitz domains the compactness of Sobolev embeddings as well as their entropy numbers  are well known.

\begin{proposition}\label{prop-spaces-dom}
	Let $\Omega\subset\rn$ be an arbitrary bounded domain, 
	$s_i\in\real$, and  $0<p_i,q_i\leq\infty$ ($p_i<\infty$ if $A$=$F$), $i=1,2$. Then 
	\begin{equation}\label{id_Omega}
		\id_\Omega : \Ae(\Omega) \to \Az(\Omega)
	\end{equation}
	is compact, if, and only if,
	\begin{equation}\label{id_Omega-comp}
		s_1-s_2 > \nd\left(\frac{1}{p_1}-\frac{1}{p_2}\right)_+\ .
	\end{equation}
Moreover
\begin{equation}\label{id_Omega-comp-ek}
	e_k(\id_\Omega) \sim k^{-\frac{s_1-s_2}{\nd}},\quad k\in \N .
\end{equation}
\end{proposition}

For bounded $C^\infty$ domains this goes back to \cite[Section~3.3]{ET} (and the references given there), for the result on general bounded domains we refer to \cite[Proposition 4.33]{T08}. Nowadays everything can be reduced to wavelet representations and mapping properties of related sequence spaces, cf. \cite[Theorem~1.20]{T08}. In this way also the necessity of \eqref{id_Omega-comp} for the compactness of $\id_\Omega$ can be shown. This is more or less obvious but not explicitly mentioned in the quoted literature.\\

Now we formulate the analogous result for spaces of generalised smoothness. To prove this we use the wavelet decomposition described in Section~\ref{wavdecomp}, cf. \eqref{wavebound}. So first we recall  known theorems concerning the compactness of embeddings of corresponding sequence spaces.  

The following abbreviations will be useful in describing compactness conditions 
\begin{equation} \frac{1}{p^\ast} = \Big(\frac{1}{p_2} - \frac{1}{p_1}\Big)_+  \quad \text{and}\quad \frac{1}{q^*} = \Big(\frac{1}{q_2} - \frac{1}{q_1}\Big)_+ ~~,
\end{equation}
where $0<p_i,q_i\leq\infty$, $i=1,2$. Moreover, recall that $c_0$ denotes the subspace of $\ell_\infty$ containing the null sequences.

\begin{proposition}[{\cite[Theorem 2]{Leo-00}}]  \label{compact-seq}
  Let $0<p_1,p_2\leq\infty$, $0<q_1,q_2\leq\infty$, $(\beta_j)_{j\in\no}$ be an arbitrary weight sequence and $(M_j)_{j\in\no}$ be a sequence of natural numbers. Then the embedding
  \begin{equation}\label{id_beta}
\id_\beta : \ell_{q_1}\left(\beta_j \ell_{p_1}^{M_j}\right) \rightarrow 
\ell_{q_2}\left(\ell_{p_2}^{M_j}\right)
  \end{equation}
  is compact if, and only if,
\begin{align}\label{cond-comp-id_beta}
 	 \ds  \Big(\beta_j^{-1} M_j^{\frac{1}{p^\ast}}\Big)_{j\in\no} \in \ell_{q^\ast}, 
 	\end{align}
where for $q^\ast=\infty$ the space $\ell_\infty$ has to be replaced by $c_0$. 
\end{proposition}

We recall that  
a weight sequence $\sigma=(\sigma_j)_{j\in\N_0}$ is called almost strongly increasing if there is a constant $\kappa_{0,\sigma}\in\N$ such that $ 2\sigma_j \le \sigma_k$ for every $j$ and $k$ with $ k\ge j + \kappa_{0,\sigma}$, cf. \cite{Leo-00}. Equivalently, one can say that a sequence  is almost strongly increasing if the lower Boyd index $\beta_\sigma$ of the sequence $\sigma$ is larger than zero,   $\beta_\sigma > 0$, cf. \eqref{Boyd}.

    \begin{theorem}[{\cite[Theorems 3, 4]{Leo-00}}]    \label{entropysequences}
Let $ 0 < p_1, p_2 \le \infty$, $0 < q_1 , q_2 \le \infty$, $(M_j)_{j\in\N_0}$ be an admissible almost strongly increasing sequence
of natural numbers,  $(\beta_j)_{j\in\N_0}$ be 
an admissible sequence and $\big(\beta_j M_j^{-\frac{1}{p^*}}\big)_{j\in\N_0}$ be  almost strongly increasing.
Then 
\begin{equation}
e_{2M_L} (\id_{\beta} : \llqe \rightarrow \llqz)    \sim \beta_L^{-1} M_L^{-(\frac{1}{p_1} - \frac{1}{p_2})}~~.
\end{equation}
\end{theorem}

\begin{remark}
  Let us mention that in the paper \cite{Leo-fsdona} the sequence is chosen as $M_j \sim 2^{j\nd}$, hence it also includes some limiting cases.
\end{remark}

Now we consider function spaces of generalised smoothness $B_{p,q}^{\sigma}(\Omega)$ as defined in \eqref{A-Omega}, \eqref{A-Omega-norm}.  Our main result reads as follows.

\begin{theorem}\label{thm-comp}
  Let $\Omega$ be an arbitrary bounded domain in $\Rn$  and  $ 0 < p_1, p_2  \le   \infty$, $0 < q_1 , q_2 \le \infty$. Let $\sigma,\tau\in \mathcal{V}$ and   $\sigma_j=\sigma(2^j)$,  $\tau_j=\tau(2^j)$,  $j\in\N_0$.
  \begin{itemize}
	\item[\upshape\bfseries{(i)}] The embedding
	\begin{equation}\label{compactfunction}
		\id_\Omega: \Bsigmae(\Omega)\hookrightarrow \Bsigmaz(\Omega),
	\end{equation}
	is compact if and only if  
		\begin{equation} \label{compactfunction2}
			\left( \sigma_j^{-1} \tau_j\, 2^{j\nd(\frac{1}{p_1}-\frac{1}{p_2})} 2^{j\nd{\frac{1}{p*}}}\right)_{j\in\no} \in \ell_{q*},
	\end{equation}	 
	where the space $\ell_\infty$ has to be replaced by $c_0$  if $q^*=\infty$.
	
	\item[{\upshape\bfseries (ii)}]
	Moreover, if  the sequence
		\begin{equation}\label{compactfunction2*}
 \left( \sigma_j \tau_j^{-1}  2^{-j\nd(\frac{1}{p_1}-\frac{1}{p_2})} 2^{-j\nd{\frac{1}{p*}}}\right)_{j\in\no} \quad \text{is almost strongly increasing,}
	\end{equation}	
	 that is, its lower Boyd index is positive, then we have 
	\begin{equation} \label{entropyfunction}
		e_{k} (\id_\Omega :  \Bsigmae(\Omega)\rightarrow  \Bsigmaz(\Omega))    \sim \left(\frac{\sigma(k^{1/\nd})}{\tau(k^{1/\nd})}\right)^{-1} ~~.
	\end{equation}
\end{itemize}
\end{theorem}

\begin{proof}
	\emph{Step 1.}  We choose $\delta>0$ and put $\Omega_\delta=\{x\in \Rn: \; d(x,\Omega)<\delta\}$. Let $Q$ be  a dyadic cube such that    $\supp \psi^\ell_{j,m}\subset Q$ for any element of the wavelet basis $\psi^\ell_{j,m}$ such that $\supp \psi^\ell_{j,m}\cap \Omega_\delta\not= \emptyset$.
	 
	If  $f\in \Bsigmae(\Omega)$  and $\|f|\Bsigmae(\Omega)\|\le 1$, then one can find some  $g\in \Bsigmae(\Rn)$ such that $f=g|_\Omega$,  $\|g|\Bsigmae(\Rn)\|\le 1+\delta$ and  $\lambda^\ell_{j,m}(g) =\langle g,\psi^\ell_{j,m}\rangle =0$ if $\supp \psi^\ell_{j,m}\cap \Omega_\delta = \emptyset$. Then  $\supp g\subset Q$ and in consequence 
	$\lambda^\ell_{j,m}(g)\in \ell_{q_1}\big(\sigma_j 2^{-jd/p_1} \ell_{p_1}^{M_j}\big)$, with $M_j= |Q|2^{j\nd}$.  
		
	On the other hand, one can find a number $\nu\in \N$ such that a dyadic cube $Q_{\nu,m_0}\subset \Omega$ for some $m_0\in \Zn$, and  $\supp \psi^\ell_{j,m}\subset \Omega$ if $j\ge \nu$ and  $\supp \psi^\ell_{j,m}\cap \overline{Q_{\nu,m_0}} \not= \emptyset$. 
	Let $I_\nu$ be a subset of $I$ consisting of all the triples $(\ell,j,m) $  that satisfy the following two conditions $j\ge \nu$ and  $\supp \psi^\ell_{j,m}\cap \overline{Q_{\nu,m_0}} \not= \emptyset$. 
	Then any distribution   
	\[f   =  \sum_{(\ell,j,m) \in I_\nu} \lambda^\ell_{j,m} \psi^\ell_{j,m}\]
	belongs to $\Bsigmae(\Omega)$. Please note that 
	$\#\{m:\; \supp \psi^\ell_{j,m}\cap \overline{Q_{\nu,m_0}} \not= \emptyset \} \sim 2^{(j-\nu)\nd}$. \\

	\emph{Step 2.} We prove part (i). The condition \eqref{compactfunction2} implies the compactness of the embedding
	\begin{equation} \label{compactfunction2a}
	\id_\beta : \ell_{q_1}\left(\beta_j \ell_{p_1}^{M_j}\right) \hookrightarrow
	\ell_{q_2}\left(\ell_{p_2}^{M_j}\right),\qquad  
	\beta_j = {\sigma_j}{\tau_j^{-1}} 2^{-j\nd(\frac{1}{p_1}-\frac{1}{p_2})},
\end{equation}
cf. Proposition~\ref{compact-seq}. 
Let $D_\gamma$, $\gamma=(\gamma_j)_{j\in\N_0}=(\tau_j 2^{-j\nd/p_2})_{j\in\N_0}$, denote the diagonal operator $\lambda^\ell_{j,m}\mapsto \gamma_j\lambda^\ell_{j,m}$ and let $D_{\gamma^{-1}}$ be its inverse. Then we have the following commutative diagrams  
	\begin{align} \label{CD1}
	\begin{CD}
			\ell_{q_1}\left( \sigma_j 2^{-j\nd/p_1}\ell_{p_1}^{M_j}\right) @>>> \ell_{q_2}\left( \tau_j2^{-j\nd/p_2} \ell_{p_2}^{M_j}\right) \\
		@VD_{\gamma}VV                                     @AAD_{\gamma^{-1}}A\\
		\ell_{q_1}\left( \beta_j \ell_{p_1}^{M_j}\right) @>>> 
		\ell_{q_2}\left( \ell_{p_2}^{M_j}\right)
	\end{CD}	
\intertext{and} 
	\begin{CD} \label{CD2}
	\ell_{q_1}\left( \beta_j \ell_{p_1}^{M_j}\right) @>>> 
	\ell_{q_2}\left( \ell_{p_2}^{M_j}\right)\\	
	@VD_{\gamma^{-1}} VV                                     @AAD_{\gamma}A\\
	\ell_{q_1}\left( \sigma_j 2^{-j\nd/p_1}\ell_{p_1}^{M_j}\right) @>>> \ell_{q_2}\left( \tau_j 2^{-j\nd/p_2} \ell_{p_2}^{M_j}\right),  	
\end{CD}
    \end{align}
with	$\|D_\gamma\|= \|D_{\gamma^{-1}}\|=1$. The above  argument with the diagonal operators shows that the compactness of \eqref{compactfunction2a} is equivalent to the compactness of the embedding
\begin{equation}\label{compactfunction3}
	\id : \ell_{q_1}\left(\sigma_j 2^{-j\nd/p_1} \ell_{p_1}^{M_j}\right) \hookrightarrow	\ell_{q_2}\left(\tau_j2^{-j\nd/p_2}\ell_{p_2}^{M_j}\right) .
\end{equation} 
	
	Let $f_n\in \Bsigmae(\Omega)$  be a bounded sequence. We may assume that   $\|f_n|\Bsigmae(\Omega)\|\le 1$  for any $n\in \N$. 
	Then there exists a sequence $(g_n)_n$, bounded in $\Bsigmae(\Rn)$ such that $\lambda^\ell_{j,m}(g_n)\in \ell_{q_1}\left(\sigma_j2^{-j\nd/p_1} \ell_{p_1}^{M_j}\right)$.  Moreover the sequence $\big(\lambda^\ell_{j,m}(g_n)\big)_n$ is bounded in $\ell_{q_1}\left(\sigma_j2^{-j\nd/p_1} \ell_{p_1}^{M_j}\right)$. If the condition \eqref{compactfunction2} is fulfilled, then the embedding \eqref{compactfunction3} is compact and we can find a subsequence $\big(\lambda^\ell_{j,m}(g_{n_k})\big)_k$ convergent in   $\ell_{q_2}\left( \tau_j 2^{-j\nd/p_2} \ell_{p_2}^{M_j}\right)$. Now the wavelet decomposition theorem implies that the the sequence $(g_{n_k})_k$ is convergent in $\Bsigmaz(\Rn)$ and in consequence $(f_{n_k})_k$ is convergent in $\Bsigmaz(\Omega)$. 
	
	On the other hand, if the condition \eqref{compactfunction2} does not hold, then the embedding \eqref{compactfunction3} is not compact. One can find a sequence  $\lambda(n)=(\lambda^\ell_{j,m}(n))_{j,m}$ in $\ell_{q_1}\left( \sigma_j 2^{-j\nd/p_1} \ell_{p_1}^{M_j}\right)$ that has no convergent subsequence in $\ell_{q_2}\left( \tau_j 2^{-j\nd/p_2} \ell_{p_2}^{M_j}\right)$. We may assume that $\lambda^\ell_{j,m}(n)=0$ if $j<\nu$ or  $\supp \psi^\ell_{j,m}\cap \overline{Q_{\nu,m_0}} = \emptyset$. Then the sequence
	\[g_n   =  \sum_{(\ell,j,m) \in I_\nu} \lambda^\ell_{j,m}(n)\, \psi^\ell_{j,m}\]
	is bounded in $\Bsigmae(\Rn)$ and has no subsequence that converges in $\Bsigmaz(\Rn)$. But 
	$ \|g_n|\Bsigmae(\Omega)\|=\|g_n|\Bsigmae(\Rn)\|$, and $ \|g_n|\Bsigmaz(\Omega)\|=\|g_n|\Bsigmaz(\Rn)\|$. This proves the necessity of the condition \eqref{compactfunction2}. 

   \emph{Step 3.} Now we deal with the asymptotic behaviour of the corresponding entropy numbers. We may assume that the dyadic cube we have chosen in Step~1 has a volume  bigger than $1$. Then the sequence $M_j$ is admissible and strongly increasing. It follows from the commutative diagrams \eqref{CD1} and \eqref{CD2} that 
   \begin{equation}
   	e_k\left(\id: \ell_{q_1}\left( \sigma_j 2^{-j\nd/p_1} \ell_{p_1}^{M_j}\right)\hookrightarrow \ell_{q_2}\left( \tau_j 2^{-j\nd/p_2} \ell_{p_2}^{M_j}\right)\right)\, \sim \, e_k(\id_\beta) .
      \end{equation} 
   So it follows  from Proposition~\ref{entropysequences} that 
   	\begin{align}
   		e_{2M_L}= & e_{2M_L} \left(\id : \ell_{q_1}\left( \sigma_j 2^{-j\nd/p_1} \ell_{p_1}^{M_j}\right)\hookrightarrow \ell_{q_2}\left( \tau_j 2^{-j\nd/p_2} \ell_{p_2}^{M_j}\right)\right)  \\ \sim & \, \beta_L^{-1} M_L^{-(\frac{1}{p_1} - \frac{1}{p_2})} =  \sigma_L^{-1}\tau_L~~, \nonumber
   	\end{align}
   if the sequence $(\beta_jM_j^{-1/p^*})_{j\in\N_0}$ is almost strongly increasing. We prove that in that case for any $k\in \N$ we can find $2^{k-1}$ balls of radius $r\sim\sigma(k^{1/\nd})/\tau(k^{1/\nd})$ in $\Bsigmaz(\Omega)$  that cover the unit ball of the space $\Bsigmae(\Omega)$.    
   Let $f\in \Bsigmae(\Omega)$ with $\|f | \Bsigmae(\Omega)\|\le 1$. According to Step~1 there is a function $g\in \Bsigmae(\Rn)$ with $\|g | \Bsigmae(\Rn)\|\le 1+\delta$ such that $f=g|_\Omega$. 
   Moreover there exists  a positive constant $C_\delta>0$ depending on $\delta$ such that $\|\lambda^\ell_{j,m}(g)|\ell_{q_1}\big( \sigma_j2^{-j\nd/p_1} \ell_{p_1}^{M_j}\big)\|\le C_\delta$. 
   
   Let $B(\lambda,r)$ denote the ball in the space 
   $\ell_{q_2}\left( \tau_j2^{-j\nd/p_2} \ell_{p_2}^{M_j}\right)$ and $\widetilde{B}(\lambda,r)$ denote the ball in the spaces $\ell_{q_1}\left( \sigma_j2^{-j\nd/p_1} \ell_{p_1}^{M_j}\right)$. 
   For any $\varepsilon> C_\delta e_k$ we can find sequences $\lambda(n)= (\lambda^\ell_{j,m}(n))_{\ell,j,m}$, $n=1,\ldots, 2^{k-1}$,  such that the balls $B(\lambda(n),\varepsilon)$ cover in $\ell_{q_2}\left( \tau_j 2^{-j\nd/p_2} \ell_{p_2}^{M_j}\right)$ the ball $\widetilde{B}(0, C_\delta)$. Now taking 
   \[ g_n   =  \sum_{(\ell,j,m) \in I} \lambda^\ell_{j,m}(n) \psi^\ell_{j,m}\, , \qquad n=1,\dots , 2^{k-1}, \]
   we get $2^{k-1}$ functions in  $\Bsigmaz(\Omega)$ and $2^{k-1}$ balls $B(g_n, r)$ with $r\sim \varepsilon$ that cover in  $\Bsigmaz(\Rn)$ the unit ball of the space $\Bsigmae(\Rn)$. In consequence we get $2^{k-1}$ balls in $\Bsigmaz(\Omega)$ centred  at $f_n=g_n|_\Omega$ with radius $r$ that cover the unit ball of the space $\Bsigmae(\Omega)$. This results in the estimate from above in \eqref{entropyfunction}.
   
   On the other hand, the cube $Q_\nu$ is a bounded Lipschitz domain. So there is a universal extension operator 
   \[\ext_\nu : \Bsigmae(Q_\nu)\rightarrow \Bsigmae(\Rn), \quad\text{and}\quad  \ext_\nu :\Bsigmaz(Q_\nu)\rightarrow \Bsigmaz(\Rn), \]
   cf. Section~\ref{wavdecomp}. Moreover, the space $\Bsigmae(Q_\nu)$ is isomorphic to $ \ell_{q_1}(\sigma_j2^{-j\nd/p_1} \ell_{p_1}^{M_j})$,  and $\Bsigmaz(Q_\nu)$ is isomorphic to $ \ell_{q_2}(\tau_j 2^{-j\nd/p_2} \ell_{p_2}^{M_j})$, $M_j\sim 2^{j\nd}$, recall \eqref{wavebound}. Let $T$ denote this isomorphism. Then we arrive at the following commutative diagram,
   \begin{align}
   		\begin{CD} \label{CD2a}
   		\ell_{q_1}\left( \sigma_j2^{-j\nd/p_1}\ell_{p_1}^{M_j}\right) @>{ T^{-1}}>> \Bsigmae(Q_\nu)
   		@>{\re_\Omega\circ \ext_\nu}>>  \Bsigmae(\Omega)\\	
   		@V{\id} VV                 @.                  @VV{\id_\Omega}V\\
   		\ell_{q_2}\left( \tau_j 2^{-j\nd/p_2}\ell_{p_2}^{M_j}\right) @<{T}<<  \Bsigmaz(Q_\nu)
   		@<{\re_\nu}<<  \Bsigmaz(\Omega),  	
   	\end{CD}
   \end{align} 
In view of the multiplicativity of the entropy numbers, this finally yields the estimate from below in \eqref{entropyfunction}. 
\end{proof}
\begin{remark}
	If the assumption \eqref{compactfunction2} is satisfied, but \eqref{compactfunction2*} is not, we will call such an embedding {\em limiting}. Otherwise, i.e., if \eqref{compactfunction2*} holds, the embedding will be called non-limiting. Part (ii) of the last above theorem describes the asymptotic behaviour of entropy numbers in the non-limiting case. Let us note that in the non-limiting situation the behaviour of entropy numbers of embeddings defined on arbitrary domains is the same as for embeddings defined on Lipschitz or smooth domains.  
\end{remark}


First results were obtained in special cases in 1998, still without 
the use of sequence space results and wavelet description. We now briefly list them now as examples of our general outcome.

\begin{Exam}\label{exam-Leo}
  Let for $b\in\real$, $s\in\real$, and $\sigma_j = 2^{js} (1+j)^b$
, the spaces $B^\sigma_{p,q}(\Omega)$ be denoted by $B_{p,q}^{s,b}(\Omega)$.
    Assume $-\infty < s_2 < s_1 < \infty$,
$b\in\real$, $0<p_1 \le p_2 \le \infty$, $0 < q_1,q_2 \le \infty$ and suppose that $s_1-s_2 - \nd(1/p_1 - 1/p_2) > 0$.  It was proved by Leopold in {\cite{Leo-fs}} 
that
\begin{equation}
e_k(\id: B_{p_1,q_{1}}^{s_1,b} (\Omega )\hookrightarrow B_{p_2,q_{2}}^{s_2}(\Omega)) \sim k^{-(s_1-s_2)/\nd} (1+\log k)^{-b}, \quad k\in \N.
\end{equation}
{This is what we called the non-limiting case above.} 
{In the limiting case for $p_1=p_2=p$  and $s_1=s_2=s$  
it was obtained in \cite{Leo-fs}, 
that if $q_1\leq q_2$ and $b>0$, then the embedding} 
\begin{equation}
\id: B_{p,q_{1}}^{s,b} (\Omega )\hookrightarrow B_{p,q_{2}}^s (\Omega )
\end{equation}
is compact and 
\begin{equation}
 e_k (\id: B_{p,q_{1}}^{s,b} (\Omega )\hookrightarrow
B_{p,q_{2}}^s (\Omega )) \sim \,(1+\log k)^{-b},  \quad k\in\nat.
\end{equation}
The result in case of $q_1 > q_2 $ in  \cite{Leo-fs}  
was not sharp, but finally improved by Cobos and K\"uhn \cite[Corollary 4.6]{CoKu}, such that the final result reads as
\begin{equation}
  e_k (\id: B_{p,q_{1}}^{s,b} (\Omega )\mapsto B_{p,q_{2}}^s
(\Omega )) \sim \,(1+\log k)^{-b+1/q^*}, \quad k\in\nat,
\label{ek-limLeo}
\end{equation}
in all cases. 
\end{Exam}

\begin{Exam} \label{exam-SV}
Cobos and K\"uhn considered in \cite{CoKu} also Besov spaces of generalised smoothness with $\sigma_j = 2^{js} \Psi(2^j)$,  where $\Psi$ is  a slowly varying function, this is, 
a Lebesgue measurable function $\Psi:[1,\infty) \rightarrow (0,\infty)$ so that 
$$
\lim_{t\rightarrow \infty}\frac{\Psi(st)}{\Psi(t)}=1 \quad \text{for all } s\ge 1,
$$
cf. e.g.  \cite{BGT, HM04}. In this case, the space $B_{p,q}^{\sigma}(\Omega)$ is denoted by $B_{p,q}^{s, \Psi}(\Omega)$.  Clearly $B_{p,q}^{s,b}(\Omega)=B_{p,q}^{s, \Psi_b}(\Omega)$, for $\Psi_b(t)=(1+ \log t)^b$, $b\in \real$. If $\Omega$ is a bounded domain in $\Rn$ so that there exists an extension operator from $B_{p,q}^{s, \Psi}(\Omega)$ to $B_{p,q}^{s, \Psi}(\Rn)$, which is the case if $\Omega$ is a bounded Lipschitz domain, cf. Section~\ref{wavdecomp}, then by \cite[Theorem 4.5]{CoKu} we have the following. 
For $0 < p \le \infty$, $0 < q_1, q_2 \le \infty$, $s \in \R$ and $\Psi$  an increasing slowly varying function, the embedding
$$
\id : B^{s,\Psi}_{p,q_1}(\Omega) \to B^{s}_{p,q_2}(\Omega)
$$
is compact if, and only if, $\left( \Psi(2^j)^{-1} \right)_{j\in\no} \in \ell_{q^*}$ (with $\ell_{q^*}$ replaced by $ c_0$ when $q^*=\infty$), and, moreover, 
\[ e_k \big(\id: B^{s,\Psi}_{p,q_1}(\Omega) \to B^{s}_{p,q_2}(\Omega)\big) \sim 
\begin{cases} \Psi(k^{1/\nd})^{-1} & \text{if } q_1 \le q_2, \smallskip \\
(\int_{k^{1/\nd}}^\infty  \Psi(t)^{-q^\ast} \dint t/t)^{1/q^*} & \text{if } q_1 > q_2 .
\end{cases}
\]
Beyond the example with logarithms $\Psi_b$, other examples of slowly varying functions $\Psi$ were worked out in \cite{CoKu}, cf. \cite[Corollaries 4.7, 4.8]{CoKu}. \\
The previous result concerns the limiting case. In the non-limiting case, accordingly to Theorem~\ref{thm-comp}(ii) we can state the following, Let $ 0 < p_i \le   \infty$, $0 < q_i \le \infty$, $s_i\in \real$, $\Psi_i$ be slowly varying functions, $i=1,2$, and assume
$$
s_1-s_2 > \nd\left(\frac{1}{p_1}-\frac{1}{p_2}\right)_+.
$$
Then 
$$
 e_k \big(\id: B^{s_1,\Psi_1}_{p_1,q_1}(\Omega) \to B^{s_2,\Psi_2}_{p_2,q_2}(\Omega)\big) \sim   k^{-\frac{s_1-s_2}{\nd}}  \frac{\Psi_2(k^{1/\nd})}{\Psi_1(k^{1/\nd})}, \quad k\in \N .
 $$
 Obviously this can also be understood as a generalisation of the classical result recalled in Proposition~\ref{prop-spaces-dom}.
\end{Exam}

We return to the spaces defined in Subsection~\ref{EdNet-approach} by Edmunds and Netrusov.

  \begin{example}\label{Ex-Ed-Net}
Let $\nd\in\N$, $Q=(0,1)^\nd$, $0 < q_1, q_2 \le \infty$, let $0 < p_1 < p_2 \le \infty$, put $1/p_1 - 1/p_2 = \alpha$. Then Edmunds and Netrusov found in \cite[Theorem 4.2]{EdNe} the following.
\benu[\bfseries\upshape (i)]
\item Let $\sigma, \tau \in \mathcal{V}$.  
Suppose that $1/q_2 - 1/q_1 \le - \alpha$ and for all $
k\in \N$ put
\begin{equation}
  A(k) = \sup_{u \ge k} \frac{\tau(u^{1/\nd})}{\sigma(u^{1/\nd})}u^\alpha \left( \min\left\{\frac{\log{(\frac{u}{k} + 1)}}{k},1\right\}\right)^\alpha .
  \label{Ak_EN}
  \end{equation}
Let $\id:B^{\sigma}_{p_1,q_1}(Q) \hookrightarrow B^{\tau}_{p_2,q_2}(Q)$ be the {continuous} embedding. Then, for all $k\in\nat$,  
\begin{equation}
  e_k(\id:B^{\sigma}_{p_1,q_1}(Q) \hookrightarrow B^{\tau}_{p_2,q_2}(Q)) \sim A(k) .
\label{ek_EN}
\end{equation}

\item Let $p_1, p_2$ and $\sigma$ satisfy the same conditions as in (i), let $0<q_2\leq q_1\leq\infty$, and  $\beta > 1/q^\ast$. Assume that $\tau(t) = \sigma(t)t^{-\nd\alpha}(1+\log{t})^{-\beta}$. Then, for all $k\in\N$,
  \begin{equation}
    e_k(\id:B^{\sigma}_{p_1,q_1}(Q) \hookrightarrow B^{\tau}_{p_2,q_2}(Q)) \sim \begin{cases}
      k^{-\alpha} (1+ \log k)^{-\beta + 2\alpha + 1/q^\ast} & \text{if}\ \beta > \frac{1}{q^\ast} + 2\alpha, \smallskip \\ k^{-\alpha}(1+ \log k)^{\alpha + 1/q^\ast}  & \text{if}\ \beta = \frac{1}{q^\ast} + 2\alpha, \smallskip \\ k^{-(\beta+1/q^\ast)/2}& \text{if}\ \beta < \frac{1}{q^\ast} + 2\alpha.
    \end{cases}
    \label{ek_EN-ii}
  \end{equation}
\eenu

\end{example}  
	
\begin{remark}
  Note that the assumptions in Example~\ref{Ex-Ed-Net} imply $\alpha>0$ and thus, in part (i), $q_1<q_2$, that is, $p^\ast= q^\ast=\infty$. Then, according to \eqref{ek_EN}, the embedding $\id:B^{\sigma}_{p_1,q_1}(Q) \hookrightarrow B^{\tau}_{p_2,q_2}(Q)$ is compact if, and only if, $\lim_{k\to\infty} A(k)=0$, which coincides with Theorem~\ref{thm-comp}(i) in this situation. Moreover, if  {\eqref{compactfunction2*}} is satisfied then this can be understood in the above setting as the supremum in \eqref{Ak_EN} being attained in $u=k$ which leads to the coincidence of \eqref{ek_EN} and \eqref{entropyfunction} in this case,
  \[
    e_k(\id:B^{\sigma}_{p_1,q_1}(Q) \hookrightarrow B^{\tau}_{p_2,q_2}(Q)) \sim A(k) \sim \frac{\tau(k^{1/\nd})}{\sigma(k^{1/d})}, \quad k\in\nat.
    \]

 However, without this additional   assumption \eqref{compactfunction2*}, we still have the sharp result \eqref{ek_EN}, unlike in case of Theorem~\ref{thm-comp}.
It would be interesting to consider further examples which do not satisfy \eqref{compactfunction}, but still fit in the above scheme of the Edmunds-Netrusov result \cite[Theorem 4.2]{EdNe}.

Concerning part (ii), the special setting refers to 
$ \sigma_j \tau_j^{-1}  2^{-j\nd(\frac{1}{p_1}-\frac{1}{p_2})}  = (1+j)^\beta$, $j\in\nat$, where $\beta>\frac{1}{q^\ast}$. Then \eqref{ek_EN-ii} can be understood as an extension of \eqref{ek-limLeo} from the case $\alpha=0$, that is, $p_1=p_2$, in Example~\ref{exam-Leo} to the case $\alpha>0$, that is, $p_1<p_2$, in Example~\ref{Ex-Ed-Net}. Plainly, when $\alpha=0$, the assumption $\beta>1/q^\ast$ excludes the second and third case in \eqref{ek_EN-ii}.
\end{remark}

}

\begin{remark}
	The estimates \eqref{ek_EN} and \eqref{ek_EN-ii} hold for any arbitrary bounded domain $\Omega$ in $\rn$. If $Q_1$ and $Q_2$ are cubes such that $2Q_1\subset \Omega \subset \frac{1}{2}Q_2$, then using the same general idea as in Step~3 of the proof of Theorem~\ref{thm-comp} we can show that there are constants  $C_1, C_2 >0$ such that  
	\begin{align*} 
		C_1\, e_k(\id:B^{\sigma}_{p_1,q_1}(Q_1) \hookrightarrow B^{\tau}_{p_2,q_2}(Q_1)) \,&\le\, e_k(\id:B^{\sigma}_{p_1,q_1}(\Omega) \hookrightarrow B^{\tau}_{p_2,q_2}(\Omega))\\
		& \le\, C_2\,  e_k(\id:B^{\sigma}_{p_1,q_1}(Q_2) \hookrightarrow B^{\tau}_{p_2,q_2}(Q_2)).
	\end{align*}  
But, we have  obviously $e_k(\id:B^{\sigma}_{p_1,q_1}(Q_1) \hookrightarrow B^{\tau}_{p_2,q_2}(Q_1))\,\sim\,  e_k(\id:B^{\sigma}_{p_1,q_1}(Q_2) \hookrightarrow B^{\tau}_{p_2,q_2}(Q_2))$. 
\end{remark}

\begin{remark}
  Plainly, there are several open problems here, e.g., what happens in the {\em limiting} case of a compact embedding in general, that is, referring to (ii) in Example~\ref{Ex-Ed-Net} when $q_2\leq q_1$, but we do not have the special coupling of $\sigma$ and $\tau$ by the logarithmic term -- can one describe the asymptotic behaviour of the entropy numbers more precisely? Another interesting question would be to study the situation in Theorem~\ref{thm-comp} when the lower Boyd index equals 0, but the upper one is positive. We discussed similar problems in case of limiting (continuous) embeddings on $\Rn$ in our paper \cite{HM08}. But this is out of the scope of the present paper.
\end{remark}

\begin{remark}
  Let us briefly mention that parallel studies for spaces, defined as trace spaces on so-called $(\alpha,\Psi)$ or $h$-sets can be found in \cite{Mo2,Bri-tese,Br02}. 
This concerns compactness results as well as estimates for the corresponding entropy numbers.
\end{remark}

\section{Nuclearity of embeddings}\label{sec-nuc}
First we recall some fundamentals of the concept and important results we rely on in the sequel.
Let $X,Y$ be Banach spaces, $T\in \mathcal{L}(X,Y)$ a linear and bounded operator. Then $T$ is called {\em nuclear}, denoted by $T\in\mathcal{N}(X,Y)$, if there exist elements $a_j\in X'$, the dual space of $X$, and $y_j\in Y$, $j\in\mathbb{N}$, such that $\sum_{j=1}^\infty \|a_j\|_{X'} \|y_j\|_Y < \infty$ and a nuclear representation $Tx=\sum_{j=1}^\infty a_j(x) y_j$ for any $x\in X$. Together with the {\em nuclear norm}
\[
 \nn{T}:=\inf\Big\{ \sum_{j=1}^\infty   \|a_j\|_{X'} \|y_j\|_Y:\ T =\sum_{j=1}^\infty a_j(\cdot) y_j\Big\},
  \]
  where the infimum is taken over all nuclear representations of $T$, the space $\mathcal{N}(X,Y)$ becomes a Banach space. It is obvious that any nuclear operator can be approximated by finite rank operators, hence 
  nuclear operators are, in particular, compact.

  \begin{remark}
    This concept has been introduced by Grothendieck \cite{grothendieck} and was intensively studied afterwards, cf. 
    \cite{Pie-op-2,pie-84,Pie-snumb} and also \cite{pie-history} for some history.
There exist extensions of the concept to $r$-nuclear operators, $0<r<\infty$, where $r=1$ refers to the nuclearity. It is well-known that {the family of all nuclear operators $\mathcal{N}:=\bigcup_{X,Y}\mathcal{N}(X,Y)$ is a Banach operator ideal.}
In Hilbert spaces $H_1,H_2$, the nuclear operators $\mathcal{N}(H_1,H_2)$ coincide with the trace class $S_1(H_1,H_2)$, consisting of those $T$ with singular numbers $(s_n(T))_n \in \ell_1$.
\end{remark}

  We collect some more or less well-known facts needed in the sequel.
\begin{proposition}\label{coll-nuc}
\benu[\upshape\bfseries (i)]
\item  If $X$ is an $n$-dimensional Banach space, $n\in\N$, then $\ \nn{\id:X\rightarrow X}= n$.  
\item  For any Banach space $X$ and any bounded linear operator $T:\ell^n_\infty\rightarrow X$ we have 
\[\nn{T} = \sum_{i=1}^n \|Te_i\| .\]
\item  If $T\in \mathcal{N}(X,Y)$ 
and {$S\in \mathcal{L}(Y,Y_0)$ and $R\in \mathcal{L}(X_0,X)$}, then $STR \in \mathcal{N}(X_0,Y_0)$  and 
\[ \nn{STR} \le \|S\| \|R\| \nn{T} . \] 
\eenu
\end{proposition}
  
Already in the early years there was a strong interest to find interesting examples of nuclear operators beyond diagonal operators in $\ell_p$ spaces, where a complete answer was obtained in \cite{tong}. We need a more general version of it, for vector-valued sequence spaces, and proved it in \cite{HaLeSk}. 
Let us introduce the following notation: for numbers $r_1,r_2\in [1,\infty]$, let $\tn(r_1,r_2)$ be given by 
\begin{equation}\label{tongnumber}
\frac{1}{\tn(r_1,r_2)} = \begin{cases}
    1, & \text{if}\ 1\leq r_2\leq r_1\leq \infty, \\
    1-\frac{1}{r_1}+\frac{1}{r_2}, & \text{if}\ 1\leq r_1\leq r_2\leq \infty.
  \end{cases}
\end{equation}
Hence $1\leq \tn(r_1,r_2)\leq \infty$, and 
\[ \frac{1}{\tn(r_1,r_2)}= 1-\left(\frac{1}{r_1}-\frac{1}{r_2}\right)_+ \geq \frac{1}{r^\ast}= \left(\frac{1}{r_2}-\frac{1}{r_1}\right)_+\ ,\]
with $\tn(r_1,r_2)=r^\ast $ if, and only if, $\{r_1,r_2\}=\{1,\infty\}$.

As already mentioned, an essential key in our arguments will be the remarkable result by Tong \cite{tong}, more precisely, its generalised version obtained in \cite{HaLeSk}. 

\begin{theorem}[{{\cite[Theorem 2.9]{HaLeSk}}}] 
\label{nucl-seq-sp}
  Let $1\leq p_i, q_i\leq\infty$, $i=1,2$, $(\beta_j)_{j\in\no}$ be an arbitrary weight sequence and $(M_j)_{j\in\no}$ be a sequence of natural numbers. Then the embedding
  \begin{equation}
\id_\beta : \ell_{q_1}\left(\beta_j \ell_{p_1}^{M_j}\right) \rightarrow 
\ell_{q_2}\left(\ell_{p_2}^{M_j}\right)
  \end{equation}
  is nuclear if, and only if,
  \begin{align}\label{nucseqsp}
     \left(\beta_j^{-1} M_j^{\frac{1}{\tn(p_1,p_2)}}\right)_{j\in\no} \in \ell_{\tn(q_1,q_2)},
        \end{align}
  where for $\tn(q_1,q_2)=\infty$ the space $\ell_\infty$ has to be replaced by $c_0$. In that case,
  \[
  \nn{\id_\beta} { = } \left\|   \left(\beta_j^{-1} M_j^{\frac{1}{\tn(p_1,p_2)}}\right)_{j\in\no} | \ell_{\tn(q_1,q_2)}\right\|.
  \]
\end{theorem}

\begin{remark}
In case of $M_j\equiv 1$ this coincides with Tong's result \cite{tong} dealing with diagonal operators acting between $\ell_r$ spaces.
\end{remark}

  In Proposition~\ref{prop-spaces-dom} we have recalled already the criterion for the compactness of the embedding
  \[ \id_\Omega : \Ae(\Omega) \to \Az(\Omega).\]
Triebel proved in \cite{Tri-nuclear} the following counterpart for its nuclearity.

\begin{proposition}[{\cite{CoDoKu,HaSk-nuc-weight,Tri-nuclear}}]\label{prod-id_Omega-nuc}
  Let $\Omega\subset\rn$ be a bounded Lipschitz domain, $1\leq p_i,q_i\leq \infty$ (with $p_i<\infty$ in the $F$-case), $s_i\in\real$, $i=1,2$. Then the embedding $\id_\Omega$ given by \eqref{id_Omega} is nuclear if, and only if,
  \begin{equation}
    s_1-s_2 > \nd-\nd\left(\frac{1}{p_2}-\frac{1}{p_1}\right)_+.
\label{id_Omega-nuclear}
  \end{equation}
\end{proposition}

\begin{remark}\label{rem-comp-nuc}
  The proposition is stated in \cite{Tri-nuclear} for the $B$-case only, but due to the independence of \eqref{id_Omega-nuclear} of the fine parameters $q_i$, $i=1,2$, and in view of (the corresponding counterpart in the classical case of) \eqref{embBFB} it can be extended immediately to $F$-spaces. The if-part of the above result is essentially covered by \cite{Pie-r-nuc} (with a forerunner in \cite{PiTri}). Also part of the necessity of \eqref{id_Omega-nuclear} for the nuclearity of $\id_\Omega$ was proved by Pietsch in \cite{Pie-r-nuc} such that only the limiting case $ s_1-s_2 = \nd-\nd(\frac{1}{p_2}-\frac{1}{p_1})_+$ was open for many decades. Then Edmunds, Gurka and Lang in \cite{EGL-3} (with a forerunner in \cite{EL-4}) obtained some answer in the limiting case which was then completely solved in \cite{Tri-nuclear}. 
  In \cite{Pie-r-nuc} some endpoint cases (with $p_i,q_i\in \{1,\infty\}$) were already discussed, and in our paper \cite{HaSk-nuc-weight} we were able to further extend Proposition~\ref{prod-id_Omega-nuc} in view of the borderline cases.

  For better comparison one can reformulate the compactness and nuclearity characterisations of $\id_\Omega$ in \eqref{id_Omega-comp} and \eqref{id_Omega-nuclear} as follows, involving the number $\tn(p_1,p_2)$ defined in \eqref{tongnumber}. Let $1\leq p_i,q_i\leq \infty$, $s_i\in\real$, $i=1,2$, and 
  \begin{equation}
    \delta := s_1 - \frac{\nd}{p_1}-s_2 + \frac{\nd}{p_2}.
    \label{delta}
    \end{equation}
   Then
  \begin{align*}
    \id_\Omega: \Ae(\Omega) \to \Az(\Omega) \quad  \text{is compact}\quad & \iff \quad \delta> \frac{\nd}{p^\ast}\qquad\text{and}\\
   \id_\Omega: \Ae(\Omega) \to \Az(\Omega) \quad \text{is nuclear}\quad & \iff \quad \delta > \frac{\nd}{\tn(p_1,p_2)}.
  \end{align*}
  Hence apart from the extremal cases $\{p_1,p_2\}=\{1,\infty\}$ (when $\tn(p_1,p_2)=p^\ast$) nuclearity is indeed stronger than compactness, i.e., 
$\id_\Omega: \Ae(\Omega) \to \Az(\Omega)$  \text{is compact, but not nuclear}, if, and only if, $ \frac{\nd}{p^\ast} < \delta \leq \frac{\nd}{\tn(p_1,p_2)}$.
\end{remark}
  
\begin{example} \label{exm-CDK}
  In \cite{CoDoKu} the authors dealt with the nuclearity of the embedding
\begin{equation} \label{id-b}
  \id_b: B^{s_1,b_1}_{p_1,q_1}(\Omega)\hookrightarrow B^{s_2,b_2}_{p_2,q_2}(\Omega),
\end{equation} 
where $\Omega$ is a bounded Lipschitz domain, $1\leq p_i, q_i\leq\infty$, $s_i, b_i\in\real$, $i=1,2$, recall the notation explained in Example~\ref{exam-Leo} already. They obtained a characterisation for almost all possible settings of the parameters. Roughly speaking, their findings are as follows, cf. \cite[Thm.~4.2]{CoDoKu}.
  \benu[\bfseries\upshape (i)]
\item 
  If $\delta> \frac{\nd}{\tn(p_1,p_2)}$, then $\id_b$ is nuclear.
\item
  If $\delta< \frac{\nd}{\tn(p_1,p_2)}$, then $\id_b$ is not nuclear.
\item
  In the limiting case, that is, if $\delta= \frac{\nd}{\tn(p_1,p_2)}$, then it depends on the interplay between the logarithmic smoothness parameters $b_i$ and the fine parameters $q_i$, $i=1,2$, in particular, $\id_b$ is nuclear if, and only if, ${b=b_1-b_2} >\frac{1}{\tn(q_1,q_2)}$ and one of the cases listed in \cite[Cor.~4.6]{CoDoKu} is satisfied.
  \eenu
  \end{example}

  This paper \cite{CoDoKu} was indeed the inspiration for our later findings, as we met here the first time the elegant technique to use Tong's result \cite{tong} in the sequence space version. Now we are able to generalise the result in \cite{CoDoKu} and close some gaps.

\begin{theorem}  \label{theorem-nuclearity}
  Let $\Omega\subset\rn$ be a bounded Lipschitz domain, $1\leq p_1,p_2 \le \infty$,  $1\leq q_1,q_2\leq \infty$, and  $\sigma=(\sigma_j)_{j\in\N_0}$, $\tau=(\tau_j)_{j\in\N_0}$ be admissible sequences. Then the embedding 
  \begin{equation*}
  \id^B_\Omega: \Bsigmae(\Omega)\hookrightarrow \Bsigmaz(\Omega)
  \end{equation*}
 is nuclear if, and only if,
  \begin{equation}
     \left( \sigma_j^{-1} \tau_j \, 2^{j\nd(\frac{1}{p_1}-\frac{1}{p_2})} 2^{j\nd{\frac{1}{\tn(p_1,p_2)}}}\right)_{j\in\no} \in \ell_{\tn(q_1,q_2)},
\label{id_Omega-nuclear-genelalised-cond}
  \end{equation}
    where for $\tn(q_1,q_2)=\infty$ the space $\ell_\infty$ has to be replaced by $c_0$.
\end{theorem}

\begin{proof}
Due to the isomorphism \eqref{wavebound} and the ideal property, the nuclearity of $\id^B_\Omega$ is equivalent to the nuclearity of the embedding
$$
\id: \ell_{q_1}\left(\sigma_j 2^{-j\nd/p_1} \ell_{p_1}^{M_j}\right) \hookrightarrow	\ell_{q_2}\left(\tau_j2^{-j\nd/p_2}\ell_{p_2}^{M_j}\right),
$$
which, taking into account  \eqref{CD1} and \eqref{CD2},  is equivalent to the nuclearity of 
$$
\id_\beta : \ell_{q_1}\left(\beta_j \ell_{p_1}^{M_j}\right) \hookrightarrow
	\ell_{q_2}\left(\ell_{p_2}^{M_j}\right),\qquad  
	\beta_j = {\sigma_j}{\tau_j^{-1}} 2^{-j\nd(\frac{1}{p_1}-\frac{1}{p_2})}.
$$
The desired conclusion is then a direct consequence of Theorem~\ref{nucl-seq-sp}.
\end{proof}

We compare now the achievement in Theorem~\ref{theorem-nuclearity} when  $\sigma_j = 2^{js_1}(1+j)^{b_1}$ and $\tau_j = 2^{js_2}(1+j)^{b_2}$, with $s_i, b_i\in\real$, $i=1,2$,  with the results obtained in \cite{CoDoKu}.  In this particular case, Theorem~\ref{theorem-nuclearity}  implies the following.

\begin{corollary}
Let  $\Omega\subset\rn$ be a bounded Lipschitz domain, $1\leq p_1,p_2 \le \infty$,  $1\leq q_1,q_2\leq \infty$, $s_i, b_i\in\real$, $i=1,2$. Then the embedding $\id_b$ given by \eqref{id-b} is nuclear if, and only if,
\begin{equation}\label{nuc-Blog}
   \Big( 2^{j(\delta-\frac{\nd}{\tn(p_1,p_2)})} (1+j)^{-b} \Big)_{j\in\no} \in \ell_{\tn(q_1,q_2)},
\end{equation}
where $\delta$ is given by \eqref{delta}, $ b := b_1 - b_2$, and for $\tn(q_1,q_2)=\infty$ the space $\ell_\infty$ has to be replaced by $c_0$.
\end{corollary}

\begin{remark}
Let us explicate \eqref{nuc-Blog} for better comparison with the result in 
\cite{CoDoKu}, cf. Example~\ref{exm-CDK}. Clearly, we get the following. 
  \benu[\bfseries\upshape (i)]
\item  If $\delta> \frac{\nd}{\tn(p_1,p_2)}$, then $\id_b$ is nuclear.
\item If $\delta< \frac{\nd}{\tn(p_1,p_2)}$, then $\id_b$ is not nuclear.
 \item In the limiting case, that is, if $\delta= \frac{\nd}{\tn(p_1,p_2)}$, then  $\id_b$ is nuclear if, and only if, $b>\frac{1}{\tn(q_1,q_2)}$ for all constellations of the parameters $q_1, q_2$.
  \eenu
  Therefore we recover the results obtained by  Cobos, Domínguez and  K\"{u}hn in  \cite{CoDoKu}, cf. Example~\ref{exm-CDK},  close some gaps in the limiting case and fully characterise the nuclearity of the  embedding  \eqref{id-b}.  Note that, according to Theorem~\ref{thm-comp}, $\id_b$ is compact if, and only if,
\begin{equation}\label{comp-Blog}
   \Big( 2^{j(\delta-\frac{\nd}{p^\ast})} (1+j)^{-b} \Big)_{j\in\no} \in \ell_{q^\ast}.
\end{equation}
So parallel to the situation explained at the end of Remark~\ref{rem-comp-nuc} for the classical case ($b_1=b_2=0$), also in the logarithmically disturbed case $\id_b$ the nuclearity criterion~\eqref{nuc-Blog} and its compactness criterion~\eqref{comp-Blog} become literally the same when replacing $p^\ast$  by $\tn(p_1,p_2)$ and $q^\ast$ by $\tn(q_1,q_2)$. Moreover, this also leads to the observation again, that nuclearity and compactness of $\id_b$ coincide when $\{p_1,p_2\}=\{1,\infty\}=\{q_1,q_2\}$, that is, in the extremal cases only.
\end{remark}

Similarly to  Example~\ref{exam-SV}, other examples could be given,  namely by perturbing the main smoothness parameter  with a slowly varying function. We remark that even more general weight functions belonging to $\mathcal{B}$ and/or $\mathcal{V}$ can be taken. 

  We conclude with a first result for spaces of type $F^\sigma_{p,q}(\Omega)$ which are a direct consequence of Theorem~\ref{theorem-nuclearity} and elementary embeddings, though omitting limiting situations at the moment.

\begin{corollary}\label{Cor-nuc-F}
 Let $\Omega\subset\rn$ be a bounded Lipschitz domain, $1\leq p_1,p_2 \le \infty$,  $1\leq q_1,q_2\leq \infty$, and  $\sigma=(\sigma_j)_{j\in\N_0}$, $\tau=(\tau_j)_{j\in\N_0}$ be admissible sequences. Let $\gamma=(\gamma_j)_{j\in\no}$ with
$$
 \gamma_j:= \sigma_j^{-1} \tau_j \, 2^{j\nd(\frac{1}{p_1}-\frac{1}{p_2})} 2^{j\nd{\frac{1}{\tn(p_1,p_2)}}}, \quad j\in\no,
$$  
and consider the embedding
$$ 
  \id^F_\Omega: \Fsigmae(\Omega)\hookrightarrow \Fsigmaz(\Omega). 
$$
 \benu[\bfseries\upshape (i)]
\item  If the upper Boyd index of $\gamma$  is negative, then $\id^F_\Omega$  is nuclear.
\item  If the lower Boyd index of $\gamma$  is positive, then $\id^F_\Omega$  is not  nuclear.
 \eenu
\end{corollary}

\begin{proof}
If the upper Boyd index of $\gamma$  is negative,  by \eqref{alpha-beta1} the condition \eqref{id_Omega-nuclear-genelalised-cond} is satisfied independently of the values of $q_1$ and $q_2$.  Then the nuclearity of  $\id^F_\Omega$ is a consequence of  \eqref{embBFB} restricted to $\Omega$ and Theorem~\ref{theorem-nuclearity}, since
$$
\Fsigmae(\Omega)\hookrightarrow B^{\sigma}_{p_1,\max\{p_1,q_1\}}(\Omega)  \hookrightarrow B^{\tau}_{p_2,\min\{p_2,q_2\}} (\Omega)  \hookrightarrow \Fsigmaz(\Omega).
$$
On the other hand, if the lower Boyd index of $\gamma$  is positive, the condition \eqref{id_Omega-nuclear-genelalised-cond} is never satisfied. Therefore $\id^F_\Omega$ is not nuclear, in consequence, once  again of Theorem~\ref{theorem-nuclearity} due to
$$
B^{\sigma}_{p_1,\min\{p_1,q_1\}} (\Omega)  \hookrightarrow \Fsigmae(\Omega)\hookrightarrow  \Fsigmaz(\Omega) \hookrightarrow B^{\tau}_{p_2,\max\{p_2,q_2\}}(\Omega).
$$
\end{proof}
  
  \begin{remark}  
In \cite{CoEdKu} some further limiting endpoint situations of nuclear embeddings like $\id:B^{\nd}_{p,q}(\Omega)\to L_p(\log L)_a(\Omega)$ are studied. For some weighted results see also \cite{Parfe-2} and our recent contribution \cite{HaSk-nuc-weight}. In \cite{HaLeSk} we studied nuclear embeddings of spaces on quasi-bounded domains, using similar techniques, that is, adapted wavelet decompositions and suitable sequence spaces results, recall Theorem~\ref{nucl-seq-sp}. In \cite{HaSkTri-nuc} we characterised the nuclearity of the Fourier transform acting between spaces of type $\A(\Rn)$.
\end{remark}

\bigskip~\bigskip~

{\small
\noindent
Dorothee D. Haroske\\
Institute of Mathematics \\
Friedrich Schiller University Jena\\
07737 Jena\\
Germany\\
{\tt dorothee.haroske@uni-jena.de}\\[4ex]
Hans-Gerd Leopold\\
Institute of Mathematics \\
Friedrich Schiller University Jena\\
07737 Jena\\
Germany\\
{\tt hans-gerd.leopold@uni-jena.de}\\[4ex]
Susana D. Moura\\
University of Coimbra\\
CMUC, Department of Mathematics\\
Largo D. Dinis, 3000-143 Coimbra\\
Portugal\\
\texttt{smpsd@mat.uc.pt}\\[4ex]
%
Leszek Skrzypczak\\
Faculty of Mathematics \& Computer Science\\
Adam  Mickiewicz University\\
ul. Uniwersytetu Pozna\'nskiego 4\\
61-614 Pozna\'n\\
Poland\\
{\tt lskrzyp@amu.edu.pl}
}

\end{document}